\documentclass[reqno,final]{amsart}
\usepackage{amsmath,amssymb,amsthm,amscd,verbatim}
\usepackage{color}
\usepackage{lscape}
\usepackage{fullpage}
\usepackage{stmaryrd}
\usepackage[notref,notcite,color]{showkeys}
\usepackage{epic,eepic}
\usepackage[dvips]{graphicx}
\usepackage[T1]{fontenc}

\numberwithin{equation}{section}
\usepackage{textcomp}
\usepackage{amsbsy}
\usepackage[foot]{amsaddr}

\usepackage[colorlinks=true,linkcolor=blue,citecolor=red,urlcolor=blue,pdfborder={0 0 0},pagebackref=true]{hyperref}

\newtheorem{theorem}{Theorem}[section]
\newtheorem{definition}[theorem]{Definition}
\newtheorem{lemma}[theorem]{Lemma}

\newtheorem{corollary}[theorem]{Corollary}
\newtheorem{conjecture}[theorem]{Conjecture}

\theoremstyle{remark}
\newtheorem{remark}[theorem]{Remark}

\def\dE {{\mathbb E}}

\def\cP{\mathcal{P}}
\def\cN{\mathcal{N}}

\newcommand{\cG}{\mathcal{G}}
\newcommand{\cE}{\mathcal{E}}
\newcommand{\cT}{\mathcal{T}}
\newcommand{\cGb}{\mathcal{G}^\bullet}

\newcommand{\cI}{\mathcal{I}}

\newcommand{\ABS}[1]{{{\left| #1 \right|}}} 

\newcommand{\SBRA}[1]{{{\left[#1\right]}}} 
\newcommand{\PAR}[1]{{{\left(#1\right)}}} 

\def\dN{\mathbb{N}}

\def\dP{\mathbb{P}}

\newcommand{\UGW}{{\mathrm{UGW}}}
\newcommand{\distr}{{\mathrm{distr}}}
\newcommand{\dd}{{\mathrm{d}}}
\newcommand{\IND}{{\mathbf{1}}}
\newcommand{\veps}{\epsilon}

\begin{document}

\title{Typicality and entropy of processes on infinite trees}
\author{\'Agnes Backhausz$^{1,2}$}
\address{$^1$ ELTE E\"otv\"os Lor\'and University, Budapest, Hungary, Faculty of Science and Alfr\'ed R\'enyi Institute of Mathematics} 
\address{$^2$ Alfr\'ed R\'enyi Institute of Mathematics }
\author{ Charles Bordenave$^3$} \address{$^3$Institut de Math\'ematiques de Marseille} 
\author{ Bal\'azs Szegedy$^2$}

\maketitle

\begin{abstract}
Consider a uniformly sampled random $d$-regular graph on $n$ vertices. If $d$ is fixed and $n$ goes to $\infty$ then we can relate typical (large probability) properties of such random graph to a family of invariant random processes (called "typical" processes) on the infinite $d$-regular tree $T_d$. This correspondence between ergodic theory on $T_d$ and random regular graphs is already proven to be  fruitful in both directions. This paper continues the investigation of typical processes with a special emphasis on entropy. We study a natural notion of micro-state entropy for invariant processes on $T_d$. It serves as a quantitative refinement of the notion of typicality and is tightly connected to the asymptotic free energy in statistical physics. Using entropy inequalities, we provide new sufficient conditions for typicality for edge Markov processes. We also extend these notions and results to processes on unimodular Galton-Watson random trees.
\end{abstract}

\begin{abstract}[Fr]
On considère un graphe $d$-régulier aléatoire  avec $n$ sommets uniformément distribué. Si $d$ est fixé et $n$ diverge, nous pouvous alors relié les propriétés typiques (de grande probabilité) d'un tel graphe aléatoire avec une famille de processus aléatoires invariants (dénommés processus "typiques") sur l'arbre $d$-régulier infini $T_d$. Cette correspondance entre théorie ergodique sur $T_d$ et graphes réguliers aléatoires s'est déjà révélée fructueuse dans les deux directions. Ce papier poursuit l'investigation des processus typiques avec un accent mis sur l'entropie. Nous y étudions une notion naturelle d'entropie micro-état  pour les processus invariant sur $T_d$. Elle sert de rafinement quantitatif  à la notion de typicalité et elle est intimement reliée à l'energie libre asymptotique en physique statistique. Au moyen d'inégalités entropiques, nous démontrons des nouvelles conditions suffisantes de typicalité pour des processus markovien sur les arêtes de l'arbre. Nous étendons aussi ces notions et résultats à des processus sur des arbres de Galton-Watson unimodulaires.
\end{abstract}

\section{Introduction}

\subsection{Typical processes and sofic entropy} \label{subsec:intro}

Random $d$-regular graphs have been extensively studied over the past 50 years \cite{MR1864966, MR2437174, MR3385636, MR1725006}.  Sophisticated methods from probability theory, combinatorics and statistical physics have been successfully used to uncover many of their properties  such as independence ratio, the density of a maximal cut or its spectral gap \cite{MR2238043, MR1384372,MR2437174}. 
The recently emerging theory of graph limits \cite{MR1873300,MR3177383,MR3256814,MR3012035} gives a new, limiting point of views on the subject. It turns out that many of the crucial properties of random $d$-regular graphs for $d$ fixed and $n$ going to infinity can also be studied in the framework of ergodic theory on the infinite $d$-regular tree $T_d$ \cite{typical}.  An illustration of the power of this method is the proof of the Gaussianity of the almost eigenvectors of random $d$-regular graphs \cite{eigenvector, bourbaki}. The proof is a combination of ergodic theory on $T_d$ with information theoretic methods.

In this paper we are interested in a question which is also related to the limit of random regular graphs, namely, to the family of typical processes on $T_d$ introduced in \cite{typical}. These objects arise as the local limits of vertex-colored random $d$-regular graphs (see formal definition below).  
These typical processes contain useful information on the structure of $d$-regular graphs. For example, the classical fact from \cite{MR624948} that the independence ratio is separated from $1/2$  is equivalent to the fact that the alternating coloring of the infinite $d$-regular tree is not typical. Several necessary conditions have been formulated for typical processes in the last years. Some of them are about the covariance structure, others are entropy inequalities  \cite{MR3385741, MR3729654, MR4164843}. However, general sufficient conditions for a process to be typical are  less common.

 In this paper we obtain sufficient conditions for the typicality of processes on $T_d$, by studying a new micro-state entropy. This entropy measures in some sense the number of finite approximations of a process on a random instance of a large $d$-regular graph. This entropy is tightly connected to Bowen's sofic entropy in measured group theory, see \cite{MR3966832}.

Our approach has another connection to the convergence of random graphs. While graph limit theory shows great promise in a variety of questions related to random $d$-regular graphs, it also revealed an intriguing open problem. It is believed that for fix $d$ and $n\to\infty$ we have that a random $d$-regular graph on $n$ vertices is  convergent (in probability) in the local-global sense of \cite{MR3177383} and also right-convergent in the sense of \cite{MR3256814}, this last notion of convergence relies on a deep statistical physics theory, see the monograph \cite{10.5555/1592967}. These general conjectures are a common strengthening of a large variety of conjectures many of which are already proven. For example the convergence of the independence ratio is proven in \cite{MR3161470,MR3689942} and the convergence of asymptotic free energies of a large class of statistical physics models on random $d$-regular graphs in \cite{MR3482664,MR4032871}.
 In this paper, we introduce an upper and a lower version of the micro-state entropies. Equality of these entropies imply the convergence of random regular graphs. For a family of processes we can establish this equality. This will lead to the sufficient conditions mentioned above, and this shows that our results might lead to a deep understanding of the structure of random regular graphs. 

\subsection{Definitions.} 
We now formalize the main definitions. Let $d \geq 3$ be an integer and $T_d$ the infinite $d$-regular tree (all vertices have $d$ neighbors). Let $M$ be a finite set. A process $X$ on  $T_d$ is a random variable on $M^{T_d}$. This process (or its law) is {\em invariant} if the law of $X$ is invariant by all automorphisms of the tree $T_d$.  We denote by $\mathcal I_d(M)$ the set of invariant probability measures on $M^{T_d}$.

We now define the Benjamini-Schramm topology. A pair $(G,f)$ formed by a graph $G = (V,E)$ and a map $f : V \to M$ will be called a colored graph with color set $M$. A rooted colored graph is a triple $(G,f,o)$ formed by a connected colored graph $(G,f)$ and a distinguished vertex $o$ of $G$, called the root.  Two rooted colored graphs $(G,f,o)$ and $(G',f',o')$ are isomorphic if there is an isomorphism of $G$ and $G'$ which preserves the colors and the roots. An equivalence class of rooted colored graphs is called an unlabeled rooted colored graph in combinatorics.

Unlabeled rooted graphs give the proper setup for defining a meaningful notion of convergence. It is however more convenient to work with rooted labeled graphs instead of unlabeled rooted graphs. To this end, we now define a randomized canonical graph in each equivalence class. We define the set of finite integer sequences as 
\begin{equation}\label{eq:defNf}
\dN^f = \cup_{k \geq 0} \dN^k 
\end{equation}
where $\dN^0 = \{ o \}$ and $\dN = \{1, 2,\ldots,\}$ by convention. The tree $T_d$ can be classically built on a subset of $\dN^f$ as follows. The root of the tree is $o$, its $d$-neighbors of $o$ are $V_1 = \{1,\ldots,d\} \subset \dN$, the neighbors of $i \in V_1$ are $o$ and $\{(i,1),\ldots ,(i,d-1) \} \in \dN^2$ and so on.  More generally, if $(G,o)$ is a rooted graph, the breadth-first search tree started at the root $o$, where ties between vertices are broken uniformly at random defines a random graph $(G',o)$ on a subset of $\dN^f$ whose law depends only the equivalence class of $(G,o)$: a vertex at distance $k$ from the root receives a label in $\dN^k$, if $(i_1,\ldots,i_{k-1})$ is the label of its parent in the search tree, it has the label $(i_1,\ldots,i_{k-1},j)$ if it the $j$-th offspring of its parent in the random ordering. We call this random rooted graph, the {\em randomly labeled rooted graph} associated to $(G,o)$. Conversely, we will say that a random labeled colored graph $(G,f,o)$ on a subset of $\dN^f$ is randomly labeled if its law is equal to the law of the randomly labeled rooted colored graph associated to its unlabeled rooted colored graph. By definition if $X \in \cI_d(M)$ then $(T_d,X,o)$ is randomly labeled.

Recall that a graph is locally finite if all its vertices have a finite number of neighbors.  We denote by $\cGb$ (respectively $\cGb_M$) the set of locally finite graphs (respectively locally finite colored graphs on the color set $M$) on the vertex set $\dN^f$ rooted at $o$ which are admissible in the sense that they are realizable as a breadth-first search labeling of a locally finite graph. The sets $\cGb$  and $\cGb_M$   are complete separable metric spaces when equipped with th\def\cP{\mathcal{P}}e distance $\dd (g,g') = \sum_{r \geq 0} 2^{-r} \IND_{(g)_r \ne (g')_r}$ where $\IND$ is the indicator function and $g_r$ is the restriction of $g$ to the vertices at distance at most $r$ from the root.  We denote by $\cP(\cGb_M)$ the set of probability measures on $\cGb_M$. We equip $\cP(\cGb_M)$ with  a distance, also denoted by $\dd$, which generates the weak topology on $\cP(\cGb_M)$ (for example, the L\'evy-Prohorov distance).

If $(G,f)$ is a locally finite colored graph and $v$ is a vertex of $G$ then we denote by $\distr_{G,v}(f)$ the law  in $\cP(\cGb_M)$ of the randomly labeled rooted graph $((G,f)(v),v)$ where $(G,f)(v)$ is the restriction of $(G,f)$ to the connected component  of $G$ containing $v$. The law $\distr_{G,v}$ in $\cP(\cGb)$  is defined similarly for a graph $G$ and a vertex $v$. Finally, if $ G$ is finite with vertex set $V$, we may define the probability measures in  $\cP(\cGb)$ and $\cP(\cGb_M)$:
$$
\distr_G =  \frac{1}{|V|} \sum_{v \in V} \distr_{G,v} \quad \hbox{ and } \quad \distr_G(f) =  \frac{1}{|V|} \sum_{v \in V} \distr_{G,v} (f).
$$


For integers $ n \geq d+1$ with $nd$ even, the set $\cG_n(d)$ of simple $d$-regular graphs on the vertex set $[n] = \{1, \ldots , n\}$ is not empty. For each integer $n \geq d+1$ with $nd$ even, let  $G_n$ be a uniformly distributed random graph on $\cG_n(d)$. Almost-surely, the probability distribution $\distr_{G_n}$ converges as $n$ goes to infinity to the Dirac mass at $T_d$ rooted at $o$ (it is a consequence of the fact that the  number of cycles of length $k$ in $G_n$ is $O(1)$ for any fixed $k$, see \cite{MR1864966}). It can further be checked that, a.s. if $\mu$ is an accumulation point of $\distr_{G_n}(f_n)$ for some sequence of colorings $f_n \in M^n$, then $\mu \in \mathcal I_d(M)$. This motivates the following definition introduced in \cite{typical}.

\begin{definition}[Typical process]
A measure $\mu\in \mathcal I_d(M)$ is \emph{weakly typical} if 
$$
\lim_{\veps \to 0} \limsup_{n \to \infty} \dP ( \exists f \in M^n  : \dd ( \distr_{G_n} (f) , \mu ) \leq \veps ) = 1.
$$
 It is \emph{strongly typical} if 
$$
\lim_{\veps \to 0} \lim_{n \to \infty} \dP ( \exists f \in M^n : \dd ( \distr_{G_n} (f) , \mu ) \leq \veps ) = 1.
$$
\end{definition} 
Note that it is apparent from the definition that being typical does not depend on the choice of the distance $\dd$ which generates the weak topology. We note also that the definition in \cite{typical} is slightly different but it turns out to be equivalent by using some measure concentration phenomena (see \cite[Section 5]{eigenvector}).

We may also define a notion of micro-state entropy of $\mu \in \mathcal I_d(M)$ as follows.  For $\mu \in \cI_d(M)$ and $r \geq 0$, the probability measure $\mu_{r}$ is defined as the restriction of $\mu$ to the vertices at distance at most $r$ from the root. Similarly, for $f\in M^n$, we define $\distr_{G} (f)_r$ as the restriction of the distribution $\distr_{G} (f)$ to the set of vertices at distance at most $r$ from the root. For any $\veps >0$, if $G \in \cG_n(d)$, we define
\begin{equation}\label{eq:defFG}
\mathcal F_G(\mu,r,\veps)  = \{ f  \in  M^n  : \dd ( \distr_{G} (f)_r , \mu_r) \leq \veps  \}.
\end{equation}
This is the set of coloring functions $f$ on $G$ which are $\veps$ close to  $\mu_r$ in the Benjamini-Schramm sense. Let  $G_n$ be a uniformly distributed random graph on $\cG_n(d)$ with $n \geq d-1$ and $nd $ even. Roughly speaking the sofic entropy of $\mu$ is a limit of 
\begin{equation}\label{eq:defHG}
H_{G_n} (\mu,r, \veps) = \frac 1 n \log |\mathcal F_{G_n}(\mu,r,\veps)|,
\end{equation}
in $n \to \infty$ and then in $\veps \to 0$ and $r\to \infty$.  Notice that random regular graphs could be replaced by other locally tree-like graph sequences, but as the expectation of $|\mathcal F_{G_n}(\mu,r,\veps)|$ is easier to compute in our case, we use the definition based on random regular graphs. Since $H_{G_n} (\mu, r,\veps)$ is a random variable in $\{-\infty\} \cup [0,\infty)$ some care is needed. More formally, 
we fix $0 < \alpha < 1$ and consider the following median value of $H_{G_n} (\mu,r, \veps)$,
$$
h_n ( \mu,r,\veps,\alpha) = \sup \left\{ h \in \{-\infty\} \cup [0,\infty) : \dP  \left( H_{G_n} (\mu, r,\veps) \geq  h \right) \geq \alpha  \right\}.
$$
Since $h_n$ is non-decreasing in $\veps$, we may define the upper and lower entropies of $\mu$ as 
$$
\bar h(\mu,r,\alpha) =  \lim_{ \veps \to 0}  \limsup_{n \to \infty} h_n ( \mu,r,\veps,\alpha) \quad \hbox{ and }\quad\underline h(\mu,r,\alpha) =  \lim_{ \veps \to 0}  \liminf_{n \to \infty} h_n ( \mu,r,\veps,\alpha).
$$

These entropies are extended real numbers in $\{-\infty\} \cup [0,\infty)$. This entropy can be interpreted as a version of Bowen's sofic entropy, see the survey \cite{MR3966832}.

The sofic entropies $\bar h(\mu,r,\alpha)$ and $\underline h(\mu,r,\alpha) $ do not depend on $\alpha \in (0,1)$. Note that they do not depend neither on the choice of the distance $\dd$ (in the sense that if two distances are topologically equivalent, the corresponding quantities $\bar h(\mu,r,\alpha)$ and $\underline h(\mu,r,\alpha) $ are equal). We shall prove the following.
\begin{lemma}\label{le:hconst}
Let $\mu \in \cI_d(M)$ and $r \geq 0$. The function $\alpha \mapsto (\bar h(\mu,r,\alpha),\underline h(\mu,r,\alpha))$ is constant on $(0,1)$. 
\end{lemma}

By Lemma \ref{le:hconst}, we may consider the common value of the entropies: for all $\alpha \in (0,1)$,
$$
\bar h(\mu,r)  = \bar h(\mu,r,\alpha) \quad \hbox{ and }\quad  \underline h(\mu,r) = \underline h(\mu,r,\alpha).
$$

By construction,  $\bar h(\mu,r)$ and  $\underline h(\mu,r)$ are the growth rates of the number of coloring of a random $d$-regular graph whose $r$-neighborhood is close to $\mu_r$. Finally, since $\bar h(\mu,r) $ and $\underline h(\mu,r)$ are non-increasing in $r$, we may define the {\em upper and lower sofic entropies} as 
$$
\bar h(\mu)  = \lim_{r \to \infty} \bar h(\mu,r) \quad \hbox{ and }\quad  \underline h(\mu) = \lim_{r \to \infty} \underline h(\mu,r).
$$

Taking the limit as $\alpha \to 1$, for $h \geq 0$, the inequality $\bar h(\mu) \geq h$ is equivalent to the existence of a vanishing sequence $(\veps_n)$ and  such that 
$$
\limsup_{n \to \infty}  \dP ( H_{G_n} (\mu, 1/\veps_n, \veps_n) \geq (h - \veps_n)_+ ) = 1,
$$
where $(x)_+  = x  \vee 0$,
and similarly for $\underline h(\mu)$. Since for $\nu,\mu \in \cP ( \cGb_M)$, $\dd(\nu_r,\mu_r)$ converges to $\dd(\nu,\mu)$ as $r \to \infty$, the sofic entropy is thus closely related to the typicality: 
\begin{lemma}\label{le:typent}
Let $\mu \in \cI_d(M)$. We have $\bar h(\mu) \geq 0$ (resp. $\underline h(\mu) \geq 0$) if and only if $\mu$ is weakly (resp. strongly) typical. 
\end{lemma}

This work is notably motivated the following conjecture which is connected to the notion of right convergence, see \cite{MR3256814} and Subsection \ref{subsec:factor} below.

\begin{conjecture}\label{conj:entropy}
For all $\mu \in\cI_d(M)$, we have $\underline h(\mu) = \bar h(\mu)$. In particular, $\mu$ is weakly typical if and only if it is strongly typical.
\end{conjecture}

In this work, we will compute the entropy $\underline h(\mu) = \bar h(\mu)$ for a large class of invariant measures $\mu \in \cI_d(M)$. This class is a class of processes where a second moment method can be applied. In a subsequent work, we will use more advanced statistical physics methods to refine our criterion.  This conjecture might also be related to the Guerra's interpolation method \cite{MR1957729} developed in the context of random graphs notably in \cite{MR1972121,MR3161470,MR3256814,MR3482664,MR4032871,MR3876913}.

\subsection{Annealed entropy}

If $r \geq 0$ is an integer and $S$ is a subset of $T_d$, we define $B_r (S)$ the subset of vertices of $T_d$ at distance at most $r$ from a vertex in $S$. For ease of notation, for $r \geq 1$, we define $S_{r} = B_{r} (o)$ as the ball of radius $r$ around the root of $T_d$ and $E_r = B_{r-1}  (\{o,1\})$ is the set of vertices at distance $r-1$ from the edge $\{o,1\}$ of $T_d$.

If $X$ is an invariant process with law $\mu \in \cI_d(M)$ and $r \geq 1$ integer, we set 
\begin{equation}\label{eq:defSigma}
\Sigma_r (X) = \Sigma_r (\mu) = H (X_{S_r}) - \frac d 2 H (X_{E_r}),
\end{equation}
where $X_S$ is the restriction of $X$ to the subset $S \subset T_d$ and $H$ is the usual Shannon entropy: if $Y$ is a random variable taking value a finite set $F$, then
$$
H(Y) = -  \sum_{x \in F} \dP( Y = x) \ln  \dP ( Y = x).
$$

By abusing notation, we denote entropy in the same way for both random variables and their distributions. It follows from \cite{eigenvector,MR3405616} (see Subsection \ref{subsec:SigmaM} below for details) that $\Sigma_r(\mu)$ is non-increasing in $r$. We may thus define
$$
\Sigma(\mu) = \lim_{r \to \infty} \Sigma_r(\mu).
$$

Note that the law of $X_{S_r}$ is $\mu_r$ and that the law of $X_{E_r}$ is a marginal of $\mu_r$. The quantities $\Sigma_r(\mu)$ and $\Sigma(\mu)$ will be called the annealed entropy of $\mu_r$ and $\mu$, the reason will be clear in the forthcoming Subsection \ref{subsec:SigmaC}. The following first moment bound is essentially contained in \cite{typical,MR3405616}. 

\begin{theorem}\label{th:1stmoment}
For any $\mu \in \cI_d(M)$ and integer $r \geq 1$, we have
$$
\bar h (\mu,r ) \leq \Sigma_r(\mu)  \quad \hbox{ and }\quad  \bar h (\mu) \leq \Sigma(\mu) . 
$$
\end{theorem}

As a corollary, we recover the "star-edge inequality" of \cite{typical,eigenvector} which is a necessary condition of typicality: if $\mu$ is weakly typical then by Lemma \ref{le:typent}, $\bar h(\mu) \geq 0$ and thus, by Theorem \ref{th:1stmoment},  $\Sigma(\mu) \geq 0$.

\begin{corollary}[\cite{typical,eigenvector}]
If $\mu \in \cI_d(M)$ is a weakly typical process then $\Sigma(\mu) \geq 0$. 
\end{corollary}

The main result of this paper is a matching lower bound for a large class of invariant processes. To this end, we first recall the notion of coupling restricted to our setting. Let $M_1$ and $M_2$ be two finite sets, $X_1$ and $X_2$ be two random variables on $M_1^{T_d}$ and $M_2^{T_d}$ with respective laws $\mu_1$ and $\mu_2$. A coupling of $\mu_1$ and $\mu_2$ is a distribution $\nu$ on $(M_1 \times M_2)^{T_d}$ such that if $Y = (Y_1,Y_2)$ has law $\nu$, $Y_i$ has law $\mu_i$ for $i = 1,2$. If $X_i$ is an invariant process for $i=1,2$, we say that $\nu$ or $Y$ is an invariant coupling if $\nu \in \cI_d(M_1 \times M_2)$.

\begin{theorem}\label{th:2ndmoment}
Let $\mu \in \cI_d(M)$.  For any  integer $r \geq 1$, if all invariant couplings $\nu$ of $\mu$ and $\mu$ satisfy $\Sigma_r(\nu) \leq 2\Sigma_r(\mu)$ then $\underline h (\mu,r) = \bar h (\mu,r)  = \Sigma_r(\mu).$
In particular, if the above condition holds for an increasing sequence of integers $(r_k)_{k \geq 1}$ then 
$
\underline h (\mu) = \bar h (\mu)  = \Sigma(\mu).
$
\end{theorem}

We note that the bound $\Sigma_r(\nu) \leq 2\Sigma_r(\mu)$ is attained for the independent coupling $Y = (X_1,X_2)$ with $X_i$ independent with law $\mu$. 
Note also that if $\nu$ is the trivial coupling of $\mu$ and $\mu$, that is $Y = (X,X)$ with $X$ with law $\mu$, we find $\Sigma_r(\nu) = \Sigma_r(\mu)$. Hence, under the condition of Theorem \ref{th:2ndmoment}, we have  $\Sigma_r(\mu) \leq 2 \Sigma_r(\mu)$ or equivalently $\Sigma_r(\mu) \geq 0$. As a corollary, by Lemma \ref{le:typent}, we thus obtain the following sufficient condition for typicality. 

\begin{corollary}
Let $\mu \in \cI_d(M)$ and $(r_k)_{k \geq 1}$ an increasing sequence of integers be such that for all invariant couplings $\nu$ of $\mu$ and $\mu$ and all $k\geq 1$ we have $\Sigma_{r_k}(\nu) \leq 2\Sigma_{r_k}(\mu)$. Then $\mu$ is strongly typical. 
\end{corollary}

\subsection{Edge-Markov processes}

There is a specific class of processes in $\cI_d(M)$ for which it is possible to improve on Theorem \ref{th:2ndmoment}, the edge-Markov processes defined as follows. As above, for integer $r \geq 0$, $B_r(S)$ is the $r$-neighborhood of a subset $S$ in $T_d$. For $r \geq 1$, recall that $S_r = B_r (o)$ and $E_r = B_{r-1} (\{o,1\})$.

\begin{definition}[Edge-Markov process] \label{def:markov} A probability measure in $M^{T_d}$ is edge-Markov if conditioned on the value at an edge, the processes on the left and right subtrees of that edge are independent. 

More generally, for integer $r \geq 1$, a probability measure on $M^{T_d}$ is $r$-Markov if conditioned on the value at $B_{r-1}(e)$, the $(r-1)$-neighborhood of an edge $e$, the processes on the left and right subtrees of $e$ are independent (for $r=1$ we recover the edge-Markov process).
 \end{definition}

Let $\cI_{d,r}(M)$ denote the set of probability measures on $M^{S_r}$ that are invariant by automorphisms of $S_r$ and whose restriction to $E_r$ is invariant by switching the two sides of the edge $\{o,1\}$. If $\mu \in \cI_d(M)$, then, $\mu_r$, its restriction to $S_r$,  is in $\cI_{d,r}(M)$. Conversely, the following lemma is easy to see.

\begin{lemma} Let $r \geq 1$ be an integer and $p \in \cI_{d,r}(M)$. Then there is a unique $r$-Markov process $\mu(p)\in\mathcal{I}_d(M)$  such that the marginal of $\mu(p)$ on $S_r$ is equal to $p$.
\end{lemma} 

If $p \in\cI_{d,r}(M)$, we define  $\Sigma (p) = \Sigma_r(\mu(p)) = H(X_{S_r}) - (d/2) H(X_{E_r})$ as in Equation \eqref{eq:defSigma}, where $X$ has law $p$.  As above, if $p_1$ and $p_2$ are probability measures on $M_1^{S_r}$ and $M_2^{S_r}$, a coupling of $p_1$ and $p_2$  is a probability measure on $M_1^{S_r} \times M_2^{S_r}$ whose marginals are $p_1$ and $p_2$. The following theorem is a strengthening of Theorem \ref{th:2ndmoment} for egde-Markov processes.

\begin{theorem}\label{th:edgeMarkov}
Let $r \geq 1$ be an integer and $p \in \cI_{d,r}(M)$. If for all couplings $q \in\cI_{d,r}(M^2)$ of $p$ and $p$,  we have $\Sigma(q) \leq 2 \Sigma(p)$  then $\underline h(\mu(p)) = \bar h(\mu(p))= \Sigma(p)$ and $\mu(p)$ is strongly typical.
\end{theorem}

Theorem \ref{th:edgeMarkov} provides an easy to check criteria for typicality for edge Markov processes. In the course of the proof, we will need an important maximizing property satisfied by edge Markov processes (a closely related characterization can be found in \cite[Theorem 1.3]{MR3405616} and \cite[Lemma 10.1]{eigenvector}).

\begin{lemma} \label{th:minSigma}
Let $X \in \cI_d(M)$ and $r \geq 1$. We have 
$$
\Sigma_{r+1}(X) \leq \Sigma_r (X)
$$
with equality if and only if $X_{S_{r+1}}$ is a $r$-Markov process on $S_{r+1}$.
\end{lemma}

 Notice that we also get that $\Sigma(X)=\Sigma_r(X)$ holds for $r$-Markov processes.

\subsection{Vertex-Markov processes}

There is a subclass of edge-Markov processes for which the annealed entropy takes a particularly simple form.

\begin{definition}[Vertex-Markov process] \label{def:markovV} Let $T$ be a tree, a probability measure in $M^{T}$ is vertex-Markov if conditioned on the value at a vertex, the processes on the pending subtrees of that vertex are independent. 
\end{definition}

Let $\cI_{e}(M)$ denote the set of probability measures on $M^{E_1}$  that are invariant by switching the two sides of the edge $\{o,1\}$. If $\mu \in \cI_d(\mu)$ then its restriction to $E_1$  is in $\cI_{e}(M)$. Conversely, if $p \in \cI_{e}(M)$, there exists a unique vertex-Markov process whose restriction to $E_1$ is $p$. We denote the law of this process by $\mu(p)$. If $X \in \mathcal I_d(M)$, we define 
$$
\Sigma_e (X) = \frac d 2 H(X_{E_1}) - (d - 1) H (X_o).
$$
If $p \in \cI_{e}(M)$, we set $\Sigma(p) = \Sigma_e ( \mu(p))$.  Vertex-Markov processes satisfy the following extremal property.
\begin{lemma}\label{le:staredgetoedgevertex}
If $X \in \cI_d(M)$ then 
$$
\Sigma_1(X)  \leq \Sigma_e(X), 
$$
with equality if and only if $X_{S_1}$ is a vertex-Markov process on $S_1$. 
\end{lemma}

Combined with Theorem \ref{th:edgeMarkov}, the above lemma implies the following corollary. 

\begin{theorem}\label{th:vertexMarkov}
Let $p \in\cI_{e}(M)$. If for all couplings $q \in\cI_{e}(M^2)$ of $p$ and $p$  we have $\Sigma(q) \leq 2 \Sigma(p)$,  then $\underline h(\mu(p)) = \bar h(\mu(p))= \Sigma(p)$ and $\mu(p)$ is strongly typical.
\end{theorem}

\begin{proof}
Let $p' = \mu(p)_1 \in \cI_{d,1}(M)$ be the law of $\mu(p)$ restricted to $S_1$. Let $q'$ be an invariant coupling of $p'$ and $p'$ and let $q$ be its restriction to $E_1$. By construction $q \in\cI_{e}(M^2)$. Moreover, by Lemma \ref{le:staredgetoedgevertex}, $\Sigma(q') \leq \Sigma_1 (\mu(q')) = \Sigma_e (\mu(q)) = \Sigma(q)$. It follows that Theorem \ref{th:vertexMarkov} is a consequence of Theorem \ref{th:edgeMarkov} applied to $r=1$ and $p' = \mu(p)_1$.
\end{proof}

\subsection{Application to factor graphs and combinatorial optimization}\label{subsec:factor}
In this paragraph, we discuss a basic connection between asymptotic free energy of factor graphs and sofic entropy. This may serve an extra motivation for studying the sofic entropy.

Let $M$ be a finite set, $r \geq 1$ be an integer and let $\varphi$ be a function on the set of rooted unlabeled $M$-colored graphs of radius $r$ taking value in $(0,\infty)$. If $G \in \mathcal G_n(d)$,
$$
Z_{G} = \sum_{f \in M^n} \prod_{v =1} ^ n \varphi((G,f,v)_r ),
$$
where $(G,f,v)_r$ is the rooted colored graph associated to the ball of radius $r$ around $v$ in $G$.  We set 
$$
\psi = \ln \varphi.
$$

The asymptotic free energy is defined as the limit of 
$(1 / n) \ln Z_{G_n}$ where $G_n$ is a uniformly sampled graph in $\cG_n(d)$ (provided that the limit exists). By standard concentration inequality (see argument in Theorem \ref{th:concmain}), it is easy to check that if $G_n$ is uniformly sampled in $\cG_n(d)$, then, in probability, as $n$ goes to infinity,
\begin{equation}\label{eq:concZGn}
\frac 1 n \ln Z_{G_n} - \dE \frac 1 n \ln Z_{G_n} \to 0. 
\end{equation}

It is straightforward to express the limits of the expected free energy in terms of the entropy. If $p \in\cI_{d,r}(M)$, we set for ease of notation $\underline h(p) = \underline h (\mu(p),r)$ and similarly for $\bar h (p)$ (there are the upper and lower growth rates of the number of colorings of $G_n$ whose $r$-neighborhood is close to $p$). In the statement below, we use the notation $\langle p , \psi \rangle  =\dE \psi ( X)$ where $X$ has law $p$.

\begin{lemma}\label{le:factorgraph}
For integer $r \geq 1$ and $\psi$ as above, if $G_n$ is uniformly distributed on $\cG_n(d)$ (with $nd$ even and $n \geq d+1$),  we have 
$$
\sup_{ p \in  \cI_{d,r}(M)} \PAR{  \underline h(p) + \langle  p , \psi \rangle } \leq  \liminf_{n \to \infty} \dE \frac 1 n \ln Z_{G_n} \leq \limsup_{n\to \infty} \dE \frac 1 n \ln Z_{G_n} \leq \sup_{ p \in  \cI_{d,r}(M)} \PAR{  \bar h(p) + \langle  p , \psi \rangle }.
$$
\end{lemma}

In particular, if Conjecture \ref{conj:entropy} holds true, then it would automatically imply the convergence of the expected free energy for all functions $\psi$. Note also that Theorem \ref{th:2ndmoment} can be used to obtain a lower bound expected free energy while Theorem \ref{th:1stmoment} can be used to get an upper bound. With proper technical conditions, it is possible to extend Lemma \ref{le:factorgraph} to some hard-constrained models, that is to some functions $\varphi = e^{\psi}$ which take value in $[0,\infty)$. For simplicity, we will however not discuss in details this possibility here.

In the same vein, in combinatorial optimization problems, we are often interested in the computation of a graph functional of the form:
$$
L_G = \max_{ f \in M^n} \sum_{v =1}^n \psi((G,f,v)_r ),
$$
with $r \geq 1$ and $\psi$ as above. Again, it is easy to check that if $G_n$ is uniformly sampled in $\cG_n(d)$, we have, in probability,
\begin{equation*}\label{eq:concLGn}
\frac {L_{G_n}}{ n} - \frac {\dE L_{G_n}}{ n}\to 0. 
\end{equation*}
The following statement is a corollary of Lemma \ref{le:factorgraph}. It shows that typical processes are intimately connected to the computation of limits of $\dE L_{G_n} / n$.
\begin{lemma}\label{le:copti}
For integer $r \geq 1$ and $\psi$ as above, if $G_n$ is uniformly distributed on $\cG_n(d)$ (with $nd$ even and $n \geq d+1$),  we have 
$$
\sup_{ p \in  \cI_{d,r}(M) : \underline h(p)  \geq 0}   \langle  p , \psi \rangle \leq  \liminf_{n\to \infty}\frac {\dE L_{G_n}}{ n} \leq \limsup_{n\to \infty} \frac {\dE L_{G_n}}{ n} \leq \sup_{ p \in  \cI_{d,r}(M) : \bar h(p) \geq 0}   \langle  p , \psi \rangle .
$$
\end{lemma}

Again, we observe that Theorem \ref{th:2ndmoment} can be used to obtain a lower bound on $\dE L_{G_n} / n$ and Theorem \ref{th:1stmoment} an upper bound.

\begin{remark}We conclude this paragraph by mentioning that, for $r = 1$, there is a simplification of Lemma \ref{le:factorgraph} and Lemma \ref{le:copti} for functions $\psi$ of the form $$\psi((G,f,v)_1) = \psi_0 ( f(v)) + \sum_{ u : u \sim v} \psi_1 ( f(v),f(u))$$
where the supremum in Lemma \ref{le:factorgraph} and Lemma \ref{le:copti} is taken over $p \in \cI_{e}(M)$ instead of $p \in  \cI_{d,1}(M)$ and the entropic term is given by $\underline h(p) = \sup \underline h (q)$ where the supremum is over all $q  \in \cI_{d,1}(M) $ whose restriction to  $E_1$ is $ p$, and similarly for $\bar h(p)$. This can be useful because it reduces the dimension of the underlying optimization problem. In that case, Theorem  \ref{th:vertexMarkov} can be used to give  lower bounds. \end{remark}

\subsection*{Organization of the paper} The remainder of this text is organized as follows. In Section \ref{sec:prop}, we will establish the key properties of sofic and annealed entropies. In Section \ref{sec:proofsMain}, we will prove the main results of this paper. In the final Section \ref{sec:extended}, we will extend our framework and main results to invariant processes on unimodular Galton-Watson trees. 

\section{Properties of sofic and annealed entropies} \label{sec:prop}

\subsection{Concentration of entropy: proof of Lemma \ref{le:hconst}}

Let  $G_n$ be a uniformly distributed random graph on $\cG_n(d)$ with $n \geq d-1$ and $nd $ even. Recall the definition of $H_{G_n} (\mu, r,\veps)$ in \eqref{eq:defHG}. The aim of this subsection is to establish the following concentration result.

\begin{theorem}\label{th:concmain}
Let $r \geq 0$, $\mu \in \cI_d(M)$, $h \in \{-\infty\} \cup [0,\infty)$. For all functions $\delta : (0,\infty) \to (0,\infty]$  with $\lim_{\veps \to 0} \delta(\veps) = 0$, we have the following:  if for all $\veps >0$, 
\begin{equation}\label{eq:conc1}
\limsup_{n \to \infty}  \frac 1 n \log \dP ( H_{G_n} (\mu, r,\veps) \geq h) \geq - \delta(\veps),
\end{equation}
then $\overline h(\mu,r,\alpha) \geq h$ for all $0 < \alpha < 1$. Conversely, for all functions $\delta$ as above such that  $\lim_{\veps \to 0} \veps^{-2} \delta(\veps) = 0$: if for all $\veps >0$,
\begin{equation}\label{eq:conc2}
\limsup_{n \to \infty}  \frac 1 n \log \dP ( H_{G_n} (\mu, , r , \veps) \leq h) \geq - \delta(\veps)
\end{equation}
then $\bar h(\mu,r,\alpha) \leq h$  for all $0 < \alpha < 1$. Finally, the same claims hold with $\liminf$ and $\underline h$ replacing $\limsup$ and $\bar h$.
\end{theorem}

It is immediate to check that Lemma \ref{le:hconst} is a corollary of Theorem \ref{th:concmain}. Beware of the asymmetry between the lower and upper bound. We believe that it is a caveat of our proof. It is ultimately due to the fact that $H_{G_n} (\mu,r,\veps)$ can be equal to $-\infty$.

The proof of Theorem \ref{th:concmain} makes a detour through a relaxation of the entropy.  Fix $\mu \in \cI_d(M)$ and $r \geq 0$. If $G$ is in $\cG_d (n)$, we define for $\beta >0$,  
$$
Z_G(\beta) = \sum_{f \in M^n} e^{-n \beta \dd ( \distr_{G} (f)_{r} , \mu_{r} ) }. 
$$

We start the proof of the proposition with a concentration inequality. 
\begin{lemma}
\label{le:concentration}
Let $G_n$ be uniformly distributed on $\cG_n(d)$ with $d n$ even and $n \geq d-1$. There exists a constant $C$ depending on $(d,r)$ and a deterministic number $s_n(\beta)$ depending on $(n, d, r,\beta)$ such that for any $t > 0$, we have 
$$
\dP \PAR{ \ABS{ \frac 1 n \ln Z_{G_n} (\beta)  - s_n (\beta) }  \geq t } \leq C \exp ( - n t ^2  / (C\beta)^2) .
$$
\end{lemma}
\begin{proof}
The proof follows a standard path. 
By classical contiguity results, it is enough to establish the claim for the configuration model (see Bollob\'as \cite[Section 2.4]{MR1864966}). Recall that the configuration model is the graph (with possible loops and multiple edges) obtained as follows. We attach to each vertex in $[n]$, $d$ half-edges. We sample a matching $m$ on the set $\vec E$ of $nd$ half-edges uniformly at random (recall that a matching is an involution without fixed point). Finally, we form a $d$-regular graph $G = G(m)$ by creating an edge for each pair of matched half-edges. 
 
Let us say that two matchings  $m,m'$ differ by a switch if there exists $(a,b,c,d)$ in $\vec E$ such that $m(e) = m'(e)$ for all $e \in \vec E \backslash \{a,b,c,d\}$ and $m(a) = b$, $m'(a) = c$, $m (c) = d$, $m'(b) = d$. If $m$ and $m'$ differ by a switch, we claim that for any $f \in M^n$, 
$$
\dd (\distr_{G(m)} (f )_{r} , \distr_{G(m')} (f )_{r} ) \leq C_0 \frac{4d(d-1)^{r-1}}{n} = \frac{\theta}{n},
$$
where $C_0$ is the diameter of $\cP(\cGb_M)$ for the distance $\dd$. 
Indeed, we have $\distr_{G,v} (f )_{r}  = \distr_{G',v'} (f')_{r}$ if the rooted subgraphs $(G,f,v)_r$ and $(G',f',v')_r$ are isomorphic. Notably, $\distr_{G(m),v} (f )_{r}  = \distr_{G(m'),v} (f)_{r}$ unless $v$ is at distance at most $r$ from an edge in the symmetric difference of $G(m)$ and $G(m')$.

We deduce that
$$
e^{-\beta \theta}\leq \frac{Z_{G(m')}}{Z_{G(m)}} \leq e^{ \beta \theta}
$$
and
$$
\ABS{ \ln Z_{G(m)} - \ln Z_{G(m')} } \leq  \beta \theta.
$$
From  \cite[Theorem 2.19]{MR1725006}, if $m$ is a  uniformly sampled matching on $\vec E$, we get 
$$
\dP \PAR{ \ABS{ \frac 1 n \ln Z_{G(m)}  - \dE \frac 1 n \ln Z_{G(m)}  }  \geq t } \leq 2 \exp ( - n t ^2  / (2 d \theta^2 \beta^2) ) .
$$
The conclusion follows with $s_n (\beta) = \dE \frac 1 n \ln Z_{G(m)} (\beta) $.
\end{proof}

We are ready for the proof of Theorem \ref{th:concmain}. 

\begin{proof}[Proof of Theorem \ref{th:concmain}]
Recall the definition of $\mathcal F_G(\mu,r,\veps)$ in \eqref{eq:defFG}. Since $\mu$ and $r$ are fixed, we write simply $\mathcal F_G(\veps)$ and set $F_G(\veps) = |\mathcal F_G(\veps)|$. For any $\veps >0$, we have 
\begin{equation}\label{eq:ZGFG1}
\ln Z_{G} (\beta) \geq \ln F_G(\veps) -n \beta \veps, 
\end{equation}
where we have used that $\dd( \distr_{G} (f )_{r} , \mu_{r}  ) \leq \veps$ for all $f \in  \mathcal F_G(\veps)$. The other way around, if $f \notin  \mathcal F_G(\veps)$, we have $\dd( \distr_{G} (f )_{r} , \mu_{r}  )  \geq  \veps$. Hence,
\begin{equation}\label{eq:ZGFG2}
\ln Z_G(\beta)  \leq \ln \PAR{ |M|^n e^{-n \beta \veps} +  F_G(\veps)} \leq \ln 2 +  \PAR{ \ln F_G( \veps) }\vee (n ( \ln |M| - \beta \veps)).
\end{equation}

We may now prove the first claim of the theorem. Assume that \eqref{eq:conc1} holds for some $h \geq 0$ (if $h = -\infty$, there is nothing to prove). Let $h_1 < h$ and $\mathcal E  = \{ H_{G_n}(\mu,r,\veps) \geq h \}$. On the event $\mathcal E$, from \eqref{eq:ZGFG1}, we have for all $\beta >0$, 
$$
\frac 1n \ln Z_{G_n}(\beta) \geq h - \beta \veps. 
$$
There exists $\beta_\veps$ such that, as $\veps \to 0$, $\beta_\veps \veps \to 0$, $\beta_\veps \to \infty$ and $\beta_\veps^2 \delta (\veps) \to 0$. For this choice of $\beta$, $(t_\veps / \beta_\veps)^{2} \gg \delta(\veps)$ for some $t_\veps \to 0$. It follows from Lemma \ref{le:concentration} and \eqref{eq:conc1} that for any $h_2,h_3$ such that $h_1 < h_2 < h_3 < h$, $s_n(\beta_\veps) \geq h_3$ for all $n$ large enough (depending on $\veps$), because the event $\mathcal E$ overlaps with the event from Lemma \ref{le:concentration} when $\beta_{\epsilon}$ satisfies this condition. Let $0 < \alpha < 1$. Applying again Lemma \ref{le:concentration}, we deduce that the event $\mathcal E_2 = \{ \frac 1 n \ln Z_{G_n} (\beta_\veps) \geq h_2\}$ has probability greater than $\alpha$ for all $n$ large enough.

Now, there exists $\eta_\veps$ such that, as $\veps \to 0$, $\eta_\veps \to 0$ and $\beta_\veps \eta_\veps \to \infty$ (for example $\eta_\veps = 1/\sqrt \beta_\veps$). We apply \eqref{eq:ZGFG2} with $\veps = \eta_\veps$. We get on the event $\mathcal E_2$, if $\veps$ is small enough,
$$
 H_{G_n}(\mu,r,\eta_\veps)  \geq \frac 1n \ln Z_{G_n}(\beta_\veps) - \frac 1 n \ln 2 \geq h_2 - \frac 1 n \ln 2 . 
$$  
The right-hand side is larger than $h_1$ if $n$ is large enough. We deduce that $\bar h(\mu,r,\alpha) \geq h_1$, since $h_1$ can be arbitrarily close to $h$, the first claim follows.

The second claim is proven similarly. Since $H_{G_n} ( \mu,r,\veps)$ takes value in $\{-\infty\} \cup [0,\infty)$, we have $H_{G_n} ( \mu,r,\veps) \leq -1$ if and only if $H_{G_n} ( \mu,r,\veps) = -\infty$. We may thus prove the second claim with $h \geq -1$. Assume that \eqref{eq:conc2} holds for some $\delta(\veps)$ which will be defined later on. We now set $\mathcal E  = \{ H_{G_n}(\mu,r,\veps) \leq h \}$. On the event $\mathcal E$, from \eqref{eq:ZGFG2}, we have for all $\beta >0$, 
$$
\frac 1n \ln Z_{G_n}(\beta) \leq \frac 1 n \ln 2  + h \vee (\ln |M| - \beta \veps). 
$$
If $\beta \geq \beta_\veps = (\ln |M| + 1)/ \veps$ then we get, for all $n$ large enough
$$
\frac 1n \ln Z_{G_n}(\beta) \leq \frac 1 n \ln 2  + h. 
$$
For our choice of $\delta(\veps)$, we have $\delta(\veps) \beta^2_\veps \to 0 $ as $\veps \to 0$. 
Then, if $h_1 > h_2 > h_3 > h$ and $\veps$ is small enough, we deduce by Lemma \ref{le:concentration} that
$$
s_n(\beta_\veps) \leq h_3
$$
for all $n$ large enough. We apply again Lemma \ref{le:concentration} and deduce that the event $\mathcal E_2 = \{ \frac 1 n \ln Z_{G_n} (\beta_\veps) \leq h_2\}$ has probability greater than $\alpha$ for all $n$ large enough. Finally, from \eqref{eq:ZGFG1}, we have on the event $\mathcal E_2$, 
$$
H_{G_n} ( \mu,r,\veps^2) \leq h_2 + \beta_\veps \veps^2
$$ 
The latter is less than $h_1$ for all $\veps$ small enough. The second claim follows. Obviously, the same argument works with $\liminf$ and $\underline h(\mu,r,\alpha)$.
\end{proof}

\subsection{Maximizers of the annealed entropy: proofs of Lemma \ref{th:minSigma} and Lemma \ref{le:staredgetoedgevertex}}
\label{subsec:SigmaM}

In this subsection, we prove Lemma \ref{th:minSigma} and Lemma \ref{le:staredgetoedgevertex}. If $X,Y$ are discrete random variables, we recall that the relative entropy of $X$ given $Y$ is 
$$
H(X |Y) = - \sum_{x,y} \dP ( (X,Y) = (x,y) ) \ln \dP ( X = x | Y = y),  
$$
where $\dP ( A |B) = \dP ( A \cap B ) / \dP (B)$ is the usual conditional probability (if $\dP (B) = 0$, $\dP(A |B)$ takes an arbitrary value). In other words, $H(X|Y)$ is the average over $Y$ of the entropy of the conditional law of $X$ given $Y$. We will repeatedly use that 
\begin{equation}\label{eq:propRE}
H(X,Y) = H (Y) + H (X|Y) \quad \hbox{ and } \quad H (X|(Y,Y') ) \leq H(X |Y),
\end{equation}
with equality if and only if $X$ conditioned on $Y$ is independent of $Y'$.

We start with the proof of Lemma \ref{th:minSigma}.

\begin{proof}[Proof of Lemma \ref{th:minSigma}] The following fact is useful. For a given integer $r \geq 1$, we introduce the finite set $N = M^{S_{r-1}}$ where as usual $S_{r-1} = B_{r-1}(o)$. We consider the map $\Psi$ from $M^{T_d}$ to $N^{T_d}$ which maps $x$ to $ \Psi(x)$ such that for $v \in T_d$, $\Psi(x)_v$ is the restriction of $x$ to $S_{r-1} (v)$ (composed by a given isomorphism from $S_{r-1}(v)$ to $S_{r-1}$). If $X$ is a process on $T_d$ then for all integers $t \geq 0$, we have $\Sigma_{t+r}(X) = \Sigma_{t+1}(\Psi(X))$. Moreover, if $X$ is a $r$-Markov process, then $\Psi(X)$ is an edge-Markov process.

As a byproduct, it is sufficient to prove Theorem \ref{th:minSigma} with $r=1$: we should check that 
\begin{equation*}\label{eq:S2S1}
\Sigma_2 (X)  \leq \Sigma_1 (X)
\end{equation*}
with equality if and only if $X_{S_2}$ is an edge Markov process. Note that the above inequality can be equivalently written as 
\begin{equation}\label{eq:S2S1b}
H(X_{S_2})   - H(X_{S_1}) - \frac d 2 H (X_{E_2}) +  \frac d 2 H (X_{E_1}) \leq 0.
\end{equation}

To check that \eqref{eq:S2S1b} holds, we need some extra notation. We denote by $L= \{2,\ldots,d\}$ the left side of $E_2$ along $E_1 = \{o,1\}$. We also set $L_i = \{(i,1),\ldots,(i,d-1)\}$ with $i=1 ,\ldots,d$. We have $S_1 = L \cup E_1$ and thus, from \eqref{eq:propRE}
$$
H(S_1 )  = H(E_1) + H (L |E_1), 
$$
where for ease of notation, for sets $S,T$, we write $H(S)$ and $H(S|T)$ in place of $H(X_S)$ and $H(X_S|X_T)$. 
Similarly, since $E_2 = S_1 \cup L_1$,
$$
H(E_2) = H(S_1)  + H(L_1 |S_1) = H(E_1) + H(L |E_1)+ H(L_1 |S_1). 
$$
Finally, since  $S_2$ is the disjoint union of $E_2$ and $\cup_{i=2}^d L_i$, we have,
 $$
 H (S_2) = H(E_2) + H ( \cup_{i=2}^d L_i| E_2).
 $$
The last three identities imply that Equation \eqref{eq:S2S1b} is equivalent to 
\begin{equation}\label{eq:S2S1t}
H ( \cup_{i=2}^d L_i | E_2)  - \PAR{\frac d 2 - 1 } H(L_1 | S_1) - \frac d 2 H ( L |E_1 ) \leq 0. 
\end{equation}
Using the invariance, we deduce from \eqref{eq:propRE} that 
$$
H ( \cup_{i=2}^d L_i | E_2)  \leq (d-1) H ( L_2 |E_2) 
$$
with equality if and only if there is conditional independence of the $X_{L_i}$'s given $X_{E_2}$. Now, since $E_2$ contains $S_1
$, we get 
$$
 H ( L_2 |E_2) \leq H (L_2 |S_1 ) = H (L_1 |S_1 ) ,
$$
with equality in case of conditional independence of $X_{L_2}$ and $X_{E_2 \backslash S_1}$ given $X_{S_1}$. It follows that the left-hand side of \eqref{eq:S2S1t} is upper bounded by 
\begin{equation}\label{eq:HSL}
\frac d 2 H ( L_1 | S_1 ) - \frac d 2 H ( L |E_1 ).
\end{equation}
From the invariance of $X$ by switching the two sides of $e$, we get $H ( L |E_1 ) = H(L_1 |E_1)$ and thus \eqref{eq:HSL} is equal to 
$$
\frac d 2 \PAR{ H ( L_1 | S_1 ) -   H ( L_1 |E_1 )}.
$$
Using again \eqref{eq:propRE}, since $E_1 \subset S_1$, this last expression is always non-positive with equality if and only if $X_{L_1}$ is conditionally independent of $X_{S_1}$ given $X_e$. This proves that \eqref{eq:S2S1b} holds. By considering the case of equality, it is then easy to check that it implies that $X_{S_2}$ is an edge Markov process. It concludes the proof of Lemma \ref{th:minSigma}.\end{proof}

We now prove Lemma \ref{le:staredgetoedgevertex}. 

\begin{proof}[Proof of Lemma \ref{le:staredgetoedgevertex}] Let $X \in \cI_d(M)$. From \eqref{eq:propRE}, with the notation used in the proof of Lemma \ref{th:minSigma}, we have 
$$
H(S_1 ) = H (o) + H(S_1 | o) \leq H (o) + d H ( E_1 | o ),
$$
with equality if the variables $(X_o,X_i)_{1 \leq i \leq d}$ conditioned on $X_o$ are independent. Using \eqref{eq:propRE} again, we get 
$$
H(S_1 )  \leq d H ( E_1  )  - (d-1) H( o ).
$$
So finally, $\Sigma_1(X) = H(S_1) - \frac d 2 H(E_1) \leq  \frac d 2 H ( E_1  )  - (d-1) H( o ) = \Sigma_e(X)$ as requested.
\end{proof}

\subsection{Combinatorial characterization of the annealed entropy}
\label{subsec:SigmaC}

In this subsection, we give a combinatorial interpretation of the annealed entropy $\Sigma_r(\mu)$.  Recall that $\mathcal G_n(d)$ is the set of simple $d$-regular graphs on the vertex set $[n]$. For $\mu \in \cI_d(M)$,  $r \geq 0$ integer and $\veps >0$, we define the set of colored graphs whose $r$-neighborhood is close to $\mu_r$ as
\begin{eqnarray*}
\cG_n(\mu,r,\veps) &= & \{ ( G,f) : G \in \cG_n(d) , f \in M^n , \dd ( \distr_{G} (f)_r , \mu_r) \leq \veps  \} \\
 & = & \bigsqcup_{G \in \cG_n(d) } \mathcal F_G(\mu,r,\veps),
\end{eqnarray*}
where  $ \sqcup$ is the disjoint union and $\mathcal F_G(\mu,r,\veps) $ was defined in \eqref{eq:defFG}. We then set  
\begin{equation}\label{eq:defSigman}
\Sigma_n(\mu,r,\veps) = \frac 1 n \PAR{ \log |\cG_n(\mu,r,\veps) | -  \log |\cG_n(d) | } = \frac 1 n \log \dE  | \mathcal F_{G_n} (\mu,r,\veps)|,
\end{equation}
where the expectation is with respect to the random graph $G_n$ uniformly distributed on $\cG_n(d)$. In comparison with the definition of $H_{G_n} (\mu, r  , \veps)$ in \eqref{eq:defHG}, $\Sigma_n(\mu,r,\veps)$ appears as an annealed quantity in the sense that there is an average over the randomness of $G_n$ inside the logarithm. The following theorem asserts that $\Sigma_n(\mu,r,\veps)$ is close to $\Sigma_r(\mu)$ as $n$ goes to infinity and $\veps$ goes to $0$.

\begin{theorem}\label{th:combS}
Let $\mu \in \cI_d(M)$ and $r \geq 1$ integer. We have
$$
\lim_{\veps \to 0} \liminf_{n \to \infty} \Sigma_n(\mu,r,\veps) = \lim_{\veps \to 0} \limsup_{n \to \infty} \Sigma_n(\mu,r,\veps) = \Sigma_r(\mu).
$$
\end{theorem}

\begin{proof}
One side of this identity can be found in \cite[Lemma 6.2]{eigenvector}. We will however give a proof which relies on \cite{DA19}  which is a generalization of \cite{MR3405616} to colored graphs. This is interesting because it connects \cite{MR3405616,DA19} to the entropic inequalities found in \cite{typical,eigenvector}. First, a classical result of Bender and Canfield \cite{BENDER1978296} implies that 
\begin{equation}\label{eq:fBC}
\frac 1 n \log |\cG_n(d)| = \frac d 2 \log n - s(d) - \log (d!) + o(1), 
\end{equation}
where $s(d) = d/2 - (d/2) \log d$. On the other hand, Proposition 5 and Proposition 6 in  Delgosha and Anantharam \cite{DA19} imply that 
$$
\lim_{\veps \to 0} \liminf_{n \to \infty} \PAR{ \frac 1 n  \log |\cG_n(\mu,r,\veps) | -  \frac d 2 \log n } = \lim_{\veps \to 0} \limsup_{n \to \infty} \PAR{ \frac 1 n \log |\cG_n(\mu,r,\veps) | -  \frac d 2 \log n } = J_r(\mu),
$$
where $J_r(\mu)$ has an explicit formula that we now describe (the same formula appears in \cite{MR3405616}).

We define $\widetilde \cT^\bullet_{r-1}$ as the set of unlabeled colored rooted $(d-1)$-ary trees of depth $r-1$. An element $g = (t,t') \in \widetilde \cE_r = \widetilde \cT^\bullet_{r-1} \times \widetilde \cT^\bullet_{r-1}$ can be seen as an unlabeled coloring of $E_r$  rooted on the oriented edge $(o,1)$. For $g = (t,t')$ in $\widetilde \cE_r$ and  $X$ a coloring of $T_d$, we then define  
$N_X (g)$ as the number of neighbors $v$ of the root such that $X$ restricted to $E_r (o,v) = B_{r-1} (\{o,v\})$ is isomorphic to $g$: more precisely such that the restriction of $X$ to $(d-1)$-ary tree rooted at $o$ (respectively $v$) in $E_r(o,v)\backslash \{o,v\}$ is isomorphic to $t$ (respectively $t'$). By construction 
\begin{equation}\label{eq:sumNX}
\sum_{g \in \widetilde \cE_r } N_X (g) = \deg(o) = d. 
\end{equation}
If $X$ is a random coloring of $T_d$ with $\mu$, we then define a probability measure on $\widetilde \cE_r $ by, for all $ g \in \widetilde \cE_r $:
$$
\pi_\mu (g) = \frac{\dE [N_X (g)]}{d},
$$
where the expectation is with respect to the randomness of $X$. We have
$$
J_r(\mu) = -s(d) + H(\widetilde X_{S_r} ) - \frac d 2 H (\pi_\mu) - \sum_{g \in \widetilde \cE_r } \dE [\log (N_X (g)!)],
$$
where $\widetilde X_{S_r} $  is the rooted unlabeled coloring associated to $X_{S_r}$. As a sanity check, if $M$ is a singleton, then $J_r(\mu) = - s(d) - \log (d!)$ and we retrieve Equation \eqref{eq:fBC}. Moreover, in view of Equation \eqref{eq:fBC}, the theorem follows from the claim 
\begin{equation}\label{eq:labelback}
J_r(\mu) = -s(d) - \log (d!) + H(X_{S_r} ) - \frac d 2 H (X_{E_r}).
\end{equation}
The expression \eqref{eq:labelback} is obtained by putting random labeling on an unlabeled rooted coloring and following the effect on the Shannon entropy.  We first observe that, since $X$ is invariant, for any  $ g \in \widetilde \cE_r $, we have
$$
\dP ( X_{E_r} \simeq g ) =   \frac{1}{d} \sum_{v=1}^d \dP ( X_{E_r(o,v)} \simeq g ) = \pi_\mu (g).
$$
It follows that $\pi_\mu$ is the law of $\widetilde X_{E_r}$ defined as the unlabeled coloring associated to $X_{E_r}$ rooted at the oriented edge $(o,1)$. Besides, since $X$ is invariant, $X_{E_r}$ is in one-to-one correspondence with the triple $(\widetilde X_{E_r},\sigma,\sigma')$ where, given $\widetilde X_{E_r}$, $\sigma$ and $\sigma'$ are independent and $\sigma$ is a uniform random labeling of $\widetilde X_{E_r}$ restricted to $E_r^o$, the  $(d-1)$-ary tree rooted at $o$  in $E_r \backslash \{o,1\}$, and similarly for $\sigma'$. From the relative entropy identity \eqref{eq:propRE}, we find that 
$$
H(X_{E_r}) = H(\pi_\mu) + 2 K,
$$
where $K$ is the relative entropy of $\sigma$ given $\widetilde X_{E_r^o}$.

Secondly, we observe that $X_{S_r}$ is in one-to-one correspondence with the vector $Y = (X_{E^1_r} ,\ldots,  X_{E^d_r})$ where $E^k_r$ is the    $(d-1)$-ary tree rooted at $k$  in $E_r \backslash \{o,k\}$. It follows that $H(Y) = H( X_{S_r})$. Also, if $\widetilde Y = (\widetilde X_{E^1_r } ,\ldots, \widetilde  X_{E^d_r})$, we find from what precedes and the invariance of $X$ that 
$$
 H( X_{S_r}) = H(Y) = H(\widetilde Y) + d K. 
$$
Finally, the difference between $\widetilde Y$ and $\widetilde X_{S_r}$ is that the neighbors of $o$ are ordered (or labeled) in $\widetilde Y$. We deduce from Lemma \ref{le:exchH} below that 
$$
H(\widetilde Y) = H ( \widetilde X_{S_r})  - \sum_{g \in \widetilde \cE_r} \dE [ \log (N_X(g) !)] +  \log (d!). 
$$
This concludes the proof of \eqref{eq:labelback}.
\end{proof}

In the proof of Theorem \ref{th:combS}, we have used the following elementary lemma. Recall that a vector is exchangeable if its law is invariant by any permutation of its coordinates. 

\begin{lemma}\label{le:exchH}
Let $F$ be a finite set and $Z = (Z_1, \ldots, Z_n)$ a random exchangeable vector in $F^n$. The counting measure $N_Z = \sum_{i=1}^n \delta_{Z_i}$ associated to $Z$ satisfies 
$$
H(Z) = H(N_Z) - \sum_{x \in F} \dE [ \log N_Z(x) !]+ \log n!. 
$$
\end{lemma}
\begin{proof}
We consider the equivalence class on $F^n$, $z \sim z'$ if $z$ and $z'$ are equal up to a permutation of the coordinates of $z$. We have $z \sim z'$ if and only if $N_z = N_{z'}$. Moreover, the number of vectors in the equivalence class of $z$ is given by the multinomial formula:
$$
\frac{n!}{\prod_{x \in F} N_z(x)!}. 
$$
Using the exchangeability of $Z$, we deduce that 
$$
\dP ( Z= z) = \frac {\prod_{x \in F} N_z(x)!}{n!} \sum_{z' \sim z} \dP ( Z = z') = \frac {\prod_{x \in F} N_z(x)!}{n!}  \dP ( N_Z = N_z). 
$$
It then remains to use the relative entropy formula \eqref{eq:propRE}.
\end{proof}

\section{Proofs of main results}
\label{sec:proofsMain}

\subsection{First moment method: proof of Theorem \ref{th:1stmoment}}

Let $r \geq 1$, $\veps >0$ and $G_n$ be uniformly sampled on $\cG_n(d)$. From Markov inequality, for any real $h$,  
$$
\dP ( H_{G_n} (\mu, r, \veps) \geq h ) = \dP ( |\mathcal F_{G_n} (\mu, r, \veps)| \geq e^{n h} ) \leq e^{-n h} \dE |\mathcal F_{G_n} (\mu, r, \veps)|.
$$ 
In particular, we find
$$
\frac 1 n \log \dP ( H_{G_n} (\mu, r, \veps) \geq h)  \leq \Sigma_n(\mu,r,\veps) - h.
$$
From Theorem \ref{th:combS}, we deduce the large deviations bound
$$
 \limsup_{n \to \infty}\frac 1 n \log \dP ( H_{G_n} (\mu, r, \veps) \geq h )  \leq \Sigma_r(\mu) - h + \delta(\veps),
$$
where $\delta(\veps)$ goes to $0$ as $\veps \to 0$.  If $h > \Sigma_r(\mu)$, the right-hand side of the above expression is negative for all $\veps$ small enough. We deduce in particular that for all $\veps$ small enough, $\dP ( H_{G_n} (\mu, r, \veps) \geq h ) $ converges to $0$. By Lemma \ref{le:hconst}, this proves that $\overline h(\mu,r) < h $. It concludes the proof of Theorem \ref{th:1stmoment}. \qed

\subsection{Second moment method: proof of Theorem  \ref{th:2ndmoment}}

Let $\mu \in \cI_d(M)$, $r \geq 1$ and set $p = \mu_r \in \cI_{d,r}(M)$. In view of Theorem \ref{th:1stmoment}, we should prove that
\begin{equation}\label{eq:step2}
\underline h(\mu,r) \geq \Sigma(p).
\end{equation}
For ease of notation, we write $\mathcal F_{G} (p, \veps)$ in place of $\mathcal F_{G} (\mu, r, \veps)$ (since this depends of $\mu$ only through $\mu_r = p$). The Paley-Zygmund inequality implies that 
\begin{eqnarray*}
\dP \PAR{ H_{G_n} (\mu, r, \veps) \geq   \Sigma_{n} ( \mu, r , \veps) -  \frac 1 n  } & = & \dP \PAR{   |\mathcal F_{G_n} (p, \veps)|
  \geq  e^{-1}  \dE |\mathcal F_{G_n} (p, \veps)|} \\
  &\geq & ( 1- e^{-1})^2 \frac{  \PAR{\dE |\mathcal F_{G_n} (p, \veps)| }^2}{\dE |\mathcal F_{G_n} (p, \veps)|^2} \\
& = & ( 1- e^{-1})^2 \frac{\exp(2n \Sigma_{n} ( \mu, r , \veps))}{\dE |\mathcal F_{G_n} (p, \veps)|^2}.
\end{eqnarray*}
Since $\mu_r = p$, we have $\Sigma(p) = \Sigma_r(\mu)$ and, by Theorem \ref{th:combS}, 
$$
\liminf_{n \to \infty} \Sigma_{n} ( \mu, r , \veps) \geq \Sigma(p) - \delta(\veps),
$$
where $\delta(\veps)$ goes to $0$ as $\veps \to 0$.  We deduce that if we manage to prove that 
\begin{equation}\label{eq:2ndmoment}
\limsup_{n \to \infty} \frac 1 n \log \dE |\mathcal F_{G_n} (p, \veps)|^2 \leq 2 \Sigma(p) + \delta'(\veps),
\end{equation}
where $\delta'(\veps)$ goes to $0$ as $\veps \to 0$, then we would get that 
$$
\liminf_{n \to \infty} \frac 1 n \log \dP \PAR{ H_{G_n} (\mu, r, \veps) \geq   \Sigma ( p) -  2 \delta(\veps)   } \geq - 2 \delta(\veps) - \delta'(\veps). 
$$
From Equation \eqref{eq:conc1} in Theorem \ref{th:concmain}, this would imply that 
$\underline h(\mu,r)  \geq  \Sigma(p)$ as claimed in \eqref{eq:step2}. 

It thus remains to prove Equation \eqref{eq:2ndmoment}. For concreteness, we may assume that the chosen distance $\dd$ generating the weak topology is the total variation distance. To that end, let $\veps >0$ and  $\mathcal N_{\varepsilon}$ be an $\varepsilon$-net on the set of invariant coupling $q\in \cI_{d,r}(M^2)$ of $p$ and $p$. Given a graph $G \in \cG_n(d)$, consider two colorings of $G$ with color set $M$  whose $r$-neighborhood statistics are at most at  total variation distance $\varepsilon$ from $p$.  The number of such pairs is $|\mathcal F_G (p, \varepsilon)|^2$. On the other hand, each pair is in fact a  coloring of $G$ on $M^2$. Then its $r$-neighborhood statistics is an element $q'\in\cI_{d,r} (M^2)$. Since both marginals of $q'$ are at most at total variation distance $\varepsilon$ from $p$, there is a measure $q ^*\in\cI_{d,r} (M^2)$ whose total variation distance is at most $2\varepsilon$ from $q'$ and whose both marginals are exactly $p$ (for each marginal of $q'$, there is an invariant coupling of this marginal and $p$ such that the two colorings are equal with probability $1-\veps$). Therefore there exists an element in the $\varepsilon$-net, $q \in \mathcal N_\veps$ such that  the distance of $q$ from the original pair of coloring is at most $3\varepsilon$. We conclude that 
\[
|\mathcal F_G (p, \varepsilon)|^2 \leq \sum_{q\in\mathcal N_{\varepsilon}} |\mathcal F_G (q, 3\varepsilon)|.\]
This implies that 
\begin{align*}
\dE |\mathcal F_{G_n} (p, \veps)|^2  &\leq \frac{1}{|\cG_n(d)|}\sum_{G\in \mathcal G_n(d)} \sum_{q\in \mathcal N_{\varepsilon}}  |\mathcal F_G (q, 3\varepsilon)| =\sum_{q\in \mathcal N_{\varepsilon}} \exp(n \Sigma_n(\mu(q),r,3\veps)).
\end{align*}
It follows that, 
$$
\frac 1 n \log \dE |\mathcal F_{G_n} (p, \veps)|^2 \leq \max_{q \in \mathcal N_\veps} \Sigma_n(\mu(q),r,3\veps) + \frac 1 n \log |\mathcal N_{\varepsilon}|.
$$
By Theorem \ref{th:combS}, we deduce that, for some function $\delta'(\veps)$ going to $0$ with $\veps \to 0$,
$$
\limsup_{n\to \infty} \frac 1 n \log \dE |\mathcal F_{G_n} (p, \veps)|^2 \leq \max_{q \in \mathcal N_\veps} \Sigma(q) + \delta'(\veps).
$$
By assumption for any $q \in \cI_{d,r}(M^2)$, $\Sigma(q) \leq 2 \Sigma(p)$. We thus have proved that \eqref{eq:2ndmoment} holds. \qed

\subsection{Proof of Theorem \ref{th:edgeMarkov}}

In view of Theorem \ref{th:1stmoment} and Theorem \ref{th:2ndmoment}, it remains to prove that for any $t \geq r$,
$$
\underline h (\mu(p),t) \geq  \Sigma_r(\mu(p)) = \Sigma(p).
$$
Let $q\in \cI_{d,t}(M^2)$ be an invariant coupling of $(\mu(p))_t$ and $(\mu(p))_t$. Then $q_r$ is an invariant coupling of $p$ and $p$ (since $((\mu(p))_t)_r = \mu(p)_r = p$ by construction). By Lemma \ref{th:minSigma}, we have 
$$
\Sigma(q)  =  \Sigma (\mu(q))  \leq \Sigma ( \mu(\mu(q)_r) )  = \Sigma (q_r). 
$$
By assumption $\Sigma (q_r) \leq 2 \Sigma(p)$. However, by Lemma \ref{th:minSigma}, we have $\Sigma(p) =  \Sigma ( \mu(p)_t)$. It follows that 
$$
\Sigma(q) \leq 2 \Sigma ( \mu(p)_t ). 
$$ 
From Equation \eqref{eq:step2} applied to the radius $t$, this implies that 
$$
\underline h (\mu(p) ,t)  \geq   \Sigma( \mu(p)_t).
$$
By a last application of Lemma \ref{th:minSigma}, the right-hand side of above expression is equal to $\Sigma(p)$. This concludes the proof of Theorem \ref{th:edgeMarkov}. \qed

\subsection{Application to factor graphs: proofs of Lemmas \ref{le:factorgraph} and Lemma \ref{le:copti}}

We start with the proof of Lemma \ref{le:factorgraph}. 

\begin{proof}[Proof of Lemma \ref{le:factorgraph}]
By construction, we have 
$$
Z_{G} = \sum_{f \in M^n} \prod_{v  = 1}^n e^{\psi((G,f,v)_r )} = \sum_{f \in M^n} e^{ n \langle \distr_{G} (f)_r, \psi \rangle }.$$

Let  $\veps >0$ and  $\mathcal N_{\varepsilon}$ be an $\varepsilon$-net of $ \cI_{d,r}(M)$. The function $ p \to \langle p , \psi \rangle$ being uniformly continuous (since $M$ is finite), there exists a function $\delta(\veps) \to 0$ as $\veps \to 0$ such that for any probability measure, say $q$, on rooted colored graphs of radius $r$,  if $\dd(q,p) \leq \veps$ then $|\langle q , \psi \rangle - \langle p, \psi \rangle | \leq \delta(\veps)$. If  $N_G$ is the number of colorings such that $\distr_{G} (f)$ is at distance larger than $\veps$ from $\mathcal N_{\varepsilon}$, it follows that 
\begin{eqnarray*}
Z_{G} &\leq &\sum_{ p \in \mathcal N_{\varepsilon}} |\mathcal F_{G} ( \mu(p),r,\veps) | e^{n \langle p, \psi \rangle  + n \delta(\veps)} + N_G e^{n\|\psi\|_{\infty}},\\
& \leq & |\cN_{\veps}| \max_{ p \in \cI_{d,r}(M)}  e^{n (  H_{G} (\mu(p), r, \veps)  +  \langle p, \psi \rangle + \delta (\veps))}  + N_G e^{n\|\psi\|_{\infty}}.
\end{eqnarray*}

Now, if $G_n$ is uniformly sampled on $\cG_n(d)$, then, for any fixed $\veps >0$, $\dP( N_{G_n} = 0)$ converges to $1$  (since $\distr_{G_n}$ converges in probability to a Dirac mass at $(T_d,o)$). Using \eqref{eq:concZGn} and taking the limit in $n$, we find 
$$
\limsup_{n \to \infty} \dE \frac 1 n \ln Z_{G_n} \leq \max_{ p \in \cI_{d,r}(M)} \PAR{     \bar h (  p )  +  \langle p, \psi \rangle + \delta' (\veps) }
$$
with $\delta'(\veps) \to 0$ as $\veps \to 0$. This gives the upper bound in Lemma \ref{le:factorgraph}.

For the lower bound, we write similarly, 
\begin{eqnarray*}
Z_{G} &\geq &\sum_{ p \in \mathcal N_{\varepsilon}} |\mathcal F_{G} ( \mu(p),r,\veps) | e^{n \langle p, \psi \rangle  - n \delta(\veps)},\\
& \geq & \max_{ p \in  \mathcal N_{\varepsilon}}  e^{n (  H_{G} (\mu(p), r, \veps)  +  \langle p, \psi \rangle  - \delta (\veps))}\\
& \geq & \max_{ p \in  \cI_{d,r}(M)}  e^{n (  H_{G} (\mu(p), r, \veps)  +  \langle p, \psi \rangle - 2 \delta (\veps))}.
\end{eqnarray*}
The conclusion follows easily. \end{proof}

Lemma \ref{le:copti} is a corollary of Lemma \ref{le:factorgraph}.

\begin{proof}[Proof of Lemma \ref{le:copti}]
For $\beta >0$, let $Z_G(\beta)$ be the factor graph model: 
$$
Z_{G} (\beta)= \sum_{f \in M^n} \prod_{v =1}^n e^{ n \beta \psi((G,f,v)_r)}. 
$$ 
By construction, we have 
$$
|M|^{-n} Z_{G}(\beta)  \leq e^{ \beta L_{G} } \leq  Z_{G}(\beta). 
$$
By Lemma \ref{le:factorgraph}, we find 
$$
\sup_{ p \in  \cI_{d,r}(M)} \PAR{ \frac{\underline h(p)}{\beta} +   \langle  p , \psi \rangle  - \frac{\ln |M|}{\beta}}\leq  \liminf_{n}\frac {\dE L_{G_n}}{ n} \leq \limsup_{n} \frac {\dE L_{G_n}}{ n} \leq \sup_{ p \in  \cI_{d,r}(M)}  \PAR{ \frac{\bar h(p)}{\beta} +  \langle  p , \psi \rangle }.
$$
We recall that $\bar h (p)$ and $\underline h (p)$ take value in $\{-\infty\} \cup [0,\ln |M|]$. We get 
$$
\sup_{ p \in  \cI_{d,r}(M) : \underline h(p) \geq 0} \PAR{   \langle  p , \psi \rangle  - \frac{\ln |M|}{\beta}}\leq  \liminf_{n}\frac {\dE L_{G_n}}{ n} \leq \limsup_{n} \frac {\dE L_{G_n}}{ n} \leq \sup_{ p \in  \cI_{d,r}(M): \underline h(p) \geq 0}  \PAR{  \frac{\ln |M|}{\beta} +  \langle  p , \psi \rangle }.
$$
We obtain the statement of the lemma by taking the limit $\beta \to \infty$.
\end{proof}

\section{Extension to processes on unimodular Galton-Watson trees}
\label{sec:extended}

\subsection{An extended setting}

We now discuss an extension to processes on random trees. We will focus our attention on {\em unimodular Galton-Watson trees}. In this section, we fix a probability measure $\pi$ on integers with positive and finite expectation: 
$$
d = \sum_{k=0}^\infty k \pi(k) > 0. 
$$
 We define $\hat \nu$, the size-biased version of $\nu$ as the probability measure defined by: for all integers $k \geq 0$
$$
\hat \pi(k) = \frac{ (k+1) \pi (k+1)}{d}.
$$
Then, the unimodular Galton-Watson tree with degree distribution $\pi$, is the Galton-Watson tree whose vertex set is a subset of $\dN^f$ defined in \eqref{eq:defNf} such that the root $o$ has a number of offsprings $N_o$ with distribution $\pi$ indexed by $1,\ldots, N_o$ and all other vertices $v$ have an independent number of offsprings $N_v$ with distribution $\hat \pi$ indexed by $(v,1),\ldots, (v,N_v)$. We will denote by $T$ a realization of this random tree and $\UGW(\pi)$ the law of the rooted tree $(T,o)$. We note that $(T,o)$ is randomly labeled in the sense defined in Subsection \ref{subsec:intro}.

For example, if $\pi$ is a Dirac mass at $d$ then $T$ is the $d$-regular tree. If $\pi$ is a Poisson random variable with mean $d$, then $\hat \pi = \pi$ and $T$ is a standard Galton-Watson tree with Poisson offspring distribution.

As its name suggests, the random rooted tree $T$ is unimodular. Recall that a random rooted graph $(G,o)$ is unimodular if for all non-negative functions $f$ on the set of doubly rooted graphs (a connected graph with two ordered distinguished vertices) which are invariant by isomorphisms, we have 
\begin{equation}\label{eq:defunimod}
\dE \sum_{v \in V} f (G,o,v) = \dE \sum_{v \in V} f (G,v,o).
\end{equation}
where $V$ is the vertex set of $G$ and the expectation is with respect to the randomness of $(G,o)$.

If $M$ is a finite set, an invariant process $X$ on $T$ is defined as a random colored tree $(T,X)$ such that $(T,X,o)$ is unimodular (that is, it satisfies \eqref{eq:defunimod} with $G = (T,X)$ and $f$ defined on the set of doubly rooted colored graphs which are invariant by isomorphisms). We denote by $\cI_\pi (M)$ the set of laws of $(T,X)$ with $X$ invariant colorings of $T$ on the color set $M$.

Now, in order to define  a relevant notion of sofic entropy, we need to choose the ensemble of finite graphs $G_n$ such that $\distr_{G_n}$ converges to $\UGW(\pi)$. A natural choice is  the family of uniform random graphs with a given degree sequence.  Let $d_n = (d_n(1), \ldots, d_n(n))$ be a sequence of integers, indexed by a subset of $\dN$, whose sum is even and such that 
$$
\distr_{d_n} = \frac 1 n \sum_{v=1}^n \delta_{d_n(v)}
$$
converges weakly to $\pi$. For technical simplicity, we assume that the degree sequence is uniformly bounded: 
 for some real $\Delta$,
\begin{equation}\label{eq:bddegree}
\sup_{n} \max_{1  \leq v \leq n} d_n (v) \leq \Delta.
\end{equation}
Note in particular that \eqref{eq:bddegree} implies that the support of $\pi$ is contained in $\{0,\ldots,\Delta\}$. From Erd\H{o}s-Gallai Theorem \cite{ERDOSGALLAI}, for all $n$ large enough, the set $\cG_n(d_n)$ of simple graphs with vertex set $[n] = \{1,\ldots, n\}$ such that for all $v \in [n]$, $v$  has degree $d_n(v)$ is not empty. Under these conditions, if $G_n$ is uniformly distributed on $\cG_n(d_n)$ then almost surely $\distr_{G_n}$ converges to $\UGW(\pi)$, see for example \cite{remcobook}. In the statements below, we will not repeat the above assumptions on the sequence $(d_n)$.

For a given probability measure $\mu \in \cI_\pi (M)$ (that is, $\mu$ is the law of an invariant coloring $(T,X)$), we can now reproduce the definition of weakly and typical processes and define the sofic entropy by taking limits of $H_{G_n} (\mu, r, \veps)$ defined in \eqref{eq:defHG}. We do not repeat the definitions since there are identical except that $G_n$ is now a random graph uniformly distributed on $\cG_n(d_n)$. For integer $r \geq 1$ and $0 < \alpha < 1$, we define the quantity $\bar h(\mu,r,\alpha)$ and $\underline h (\mu,r,\alpha)$ exactly as done below \eqref{eq:defHG}.  Lemma \ref{le:hconst} continues to holds in this more general setting.

\begin{lemma}\label{le:hconstT}
Let $\mu \in \cI_\pi(M)$ and $r \geq 0$. The function $\alpha \mapsto (\bar h(\mu,r,\alpha),\underline h(\mu,r,\alpha))$ is constant on $(0,1)$. 
\end{lemma}
\begin{proof}
The proof of Theorem \ref{th:concmain} works verbatim under the assumption \eqref{eq:bddegree}. 
\end{proof}

We define $\bar h (\mu,r)$ and $\underline h (\mu,r)$ as the common value of $\bar h(\mu,r,\alpha)$ and $\underline h (\mu,r,\alpha)$. The upper and lower sofic entropies $\bar h (\mu)$ and $\underline h (\mu)$ are the limits in $r$ of $\bar h (\mu,r)$ and $\underline h (\mu,r)$. Exactly as in Lemma \ref{le:typent}, as a corollary of Lemma \ref{le:hconstT}, we obtain the following claim. 
\begin{lemma}\label{le:typentT}
Let $\mu \in \cI_\pi(M)$. We have $\bar h(\mu) \geq 0$ (resp. $\underline h(\mu) \geq 0$) if and only if $\mu$ is weakly (resp. strongly) typical. 
\end{lemma}

\subsection{Annealed entropy}
In this broader setting, the annealed entropy is defined as follows. Let $(T,o)$ be a  randomly labeled rooted unimodular tree and $X$ an invariant coloring of $T$ with $(T,X)$ having law $\mu$. The degree of a vertex $v$ of $T$ is denoted by $\deg(v)$. We assume that $d = \dE \deg(o) > 0$. As above, if $r \geq 0$ is an integer and $S$ is a subset of the vertices of $T$, $B_r (S)$  is the subset of vertices of $T$ at distance at most $r$ from $S$.  For $r \geq 1$, we set $S_r = B_r(o)$ and, if $\deg(o)\geq 1$, we set $E_r = B_{r-1} (\{o,1\})$ (since $T$ is randomly labeled, the neighbors of the root are indexed by $(1,\ldots,\deg(o))$).

We denote by $X_{S_r}$ the colored tree $(T,X)$ restricted to $S_r$: by construction,  $X_{S_r}$ has law $\mu_r$.  We also need to define the law of $X$ restricted to $E_r$ but this requires a biasing of the tree $T$. This is done as follows. A (directed) edge-rooted graph is defined as a pair $(G,\rho)$ formed by a connected graph $G$ and a distinguished directed edge $\rho = (u,v)$ (that is, $\{u,v\}$ is an edge of the graph). 
Now,  we denote by $\vec \mu$ the law on colored edge-rooted trees defined by:
\begin{equation}\label{eq:vecmu}
\vec \mu ( \cdot)    = \frac 1 d \dE \SBRA{ \deg(o) \IND( (T,X, (o,1)) \in \cdot } ,
\end{equation}
where $(T,X)$ has law $\mu$ and $d = \dE \deg(o)$. Note that under the probability measure $\vec \mu$, $o$ has at least degree $1$ and thus $\{o,1\}$ is an edge of the tree. We denote by $(\vec T, \vec X, \rho)$, with $\rho = (o,1)$ a random variable with law $\vec \mu$.  It is easy to check that Equation \eqref{eq:defunimod} implies that $\vec \mu$ is invariant by switching the two sides of the oriented edge. Moreover, if $T$ has law $\UGW(\pi)$ then $\vec T$  is given by two independent Galton-Watson trees with offspring distribution $\hat \pi$ whose roots are connected by the root-edge, see \cite[Example 1.1]{aldlyo}. In particular, if $\pi$ is a Dirac mass at $d$, then $T =  \vec T$.

We denote by $\vec X_{E_r}$ the colored tree $(\vec T, \vec X)$ restricted to $E_r$. The law of $\vec X_{E_r}$ is $\vec \mu_r$, the restriction of $\vec \mu$ to $E_r$. We observe that $\vec \mu_r$ depends on $\mu$ only through its marginal $\mu_r$.

Now, if $X$ is an invariant coloring of $T$ with law $\mu \in \cI_\pi (M)$ and $r \geq 1$ is an integer, we set 
\begin{equation}\label{eq:defSigmaT}
\Sigma_r (X) = \Sigma_r (\mu) = H (X_{S_r}) - \frac d 2 H (\vec X_{E_r}) - H(\pi).
\end{equation}
See Remark \ref{rk:defSigmaT} for an alternative expression which is arguably more natural. Thanks to assumption \eqref{eq:bddegree} it is immediate that the above entropies are finite as soon as $M$ is finite. Note also that $\Sigma_r (\mu) $ depends on $\mu$ only through $\mu_r$. We will check in Lemma \ref{th:minSigmaT} below that $\Sigma_r(\mu)$ is non-increasing in $r$. We may thus define
$$
\Sigma(\mu) = \lim_{r \to \infty} \Sigma_r(\mu).
$$
The quantities $\Sigma_r(\mu)$ and $\Sigma(\mu)$ are the annealed entropies of $\mu_r$ and $\mu$. The following lemma generalizes Lemma \ref{th:1stmomentT}.

\begin{theorem}\label{th:1stmomentT}
For any $\mu \in \cI_\pi(M)$ and integer $r \geq 1$, we have
$$
\bar h (\mu,r ) \leq \Sigma_r(\mu)  \quad \hbox{ and }\quad  \bar h (\mu) \leq \Sigma(\mu) . 
$$
\end{theorem}

There is also an analog of Theorem \ref{th:2ndmoment}. 

\begin{theorem}\label{th:2ndmomentT}
Let $\mu \in \cI_\pi(M)$.  For any  integer $r \geq 1$, if all invariant couplings $\nu$ of $\mu$ and $\mu$, we have $\Sigma_r(\nu) \leq 2\Sigma_r(\mu)$ then $\underline h (\mu,r) = \bar h (\mu,r)  = \Sigma_r(\mu).$
In particular, if the above condition holds for an increasing sequence of integers $(r_k)_{k \geq 1}$ then 
$
\underline h (\mu) = \bar h (\mu)  = \Sigma(\mu).
$
\end{theorem}

\begin{remark}\label{rk:defSigmaT}
The annealed entropy is also given by the formula:
$$
\Sigma_r (X) = H (X_{S_r}|T_{S_r}) - \frac d 2 H (\vec X_{E_r}|\vec T_{E_r}),
$$
where  $H (X |Y)  = H(X,Y) - H(Y)$ is the relative entropy. Indeed $\vec T$ is the union of two independent copies of $T'$, a Galton-Watson tree with offspring distribution $\hat \pi$, while $T$ is the union of $N$ independent copies of $T'$ and with $N$ independent with distribution $\pi$. It follows that, $H ( \vec T_{E_r}) = 2 H(T'_{S_{r-1}})$ and $H( T_{S_r} ) = H (N) + d H(T'_{S_{r-1}})$ (from \eqref{eq:propRE}). In particular, $H(T_{S_r} ) - (d/2) H (\vec T_{E_r}) = H(N) = H (\pi)$.
\end{remark}

\subsection{Markov processes}

There is an extension of Theorem \ref{th:edgeMarkov} and Theorem \ref{th:vertexMarkov} in our extended setting. The previous definitions of Markov processes carry over when conditioned on the random tree. More precisely, we use the following definitions.

\begin{definition}[Markov process]\label{def:markovT} Let $(T,X)$ be a random coloring of a random tree $T$ on a finite set $M$ with law $\mu$.  For integer $r \geq 1$, $X$ or $\mu$ is $r$-Markov if conditioned on $T$ on $B_{r-1}(e)$ and on the value at $B_{r-1}(e)$, the $(r-1)$-neighborhood of an edge $e$, the processes on the left and right subtrees of $e$ are independent. Similarly, $X$ or $\mu$ is vertex-Markov if conditioned on $T$ and on the value at a vertex, the processes on the pending subtrees of that vertex are independent. 
 \end{definition}

For integer $r \geq 1$, let $\cI_{\pi,r}(M)$ denote the set of laws $\mu$ of coloring $(T',X')$ on $M$ which are randomly labeled,  such that $T'$ has law $\UGW(\pi)_r$ (the law of the restriction of $T$ to $S_r$) and  such that $\vec \mu$ as defined above is invariant  by switching the two sides of the oriented edge. If $\mu \in \cI_{\pi}(M)$ then $\mu_r \in  \cI_{\pi}(M)$. Conversely, we have the following (see \cite[Proposition 1.1]{MR3405616}):

\begin{lemma}Let $r \geq 1$ integer and $p \in\cI_{\pi,r} (M)$. Then there is a unique $r$-Markov process $\mu(p)\in\mathcal{I}_\pi(M)$  such that the marginal of $\mu(p)$ on $S_r$ is equal to $\mu$.
\end{lemma} 

If $p \in\cI_{\pi,r} (M)$, we define  $\Sigma (p) = \Sigma_r(\mu(p))$ as in Equation \eqref{eq:defSigmaT}.  The following theorem is an extension of Theorem \ref{th:edgeMarkov}.

\begin{theorem}\label{th:edgeMarkovT}
Let $r \geq 1$ be an integer and $p \in\cI_{\pi,r} (M)$. If for all couplings $q \in\cI_{\pi,r} (M^2)$ of $p$ and $p$,  we have $\Sigma(q) \leq 2 \Sigma(p)$  then $\underline h(\mu(p)) = \bar h(\mu(p))= \Sigma(p)$ and $\mu(p)$ is strongly typical.
\end{theorem}

There is a version of this theorem for vertex-Markov processes.  As above, let $\cI_{e}(M)$ denote the set of probability measures on $M^{E_1}$ that are invariant by switching the two sides of the edge $\{o,1\}$. If $\mu \in \cI_{\mu}(M)$ then the restriction of $\vec \mu$ to $E_1$ is in $\cI_{e}(M)$. Conversely, if $ p \in \cI_{e}(M)$, there exists a unique vertex-Markov process $\mu(p)$ in $\cI_{\pi}(M)$ such that $\vec \mu$ restricted to $E_1$ is in $\cI_{e}(M)$. If $(T,X) \in \mathcal I_{\pi} (M)$, we define 
$$
\Sigma_e(X) =  \frac d 2 H (\vec X_{E_1}) - d H (\vec X_o) + H( X_o) - H(\pi).
$$
If $p \in \cI_e(M)$, we set $\Sigma(p) = \Sigma_e(\mu(p))$. The following theorem is an extension of Theorem \ref{th:vertexMarkov}.

\begin{theorem}\label{th:vertexMarkovT} 
Let $p \in\cI_e(M)$. If for all couplings $q \in\cI_e(M^2)$ of $p$ and $p$,  we have $\Sigma(q) \leq 2 \Sigma(p)$  then $\underline h(\mu(p)) = \bar h(\mu(p))= \Sigma(p)$ and $\mu(p)$ is strongly typical.
\end{theorem}

In the remainder of the paper, we explain the proofs of Theorem \ref{th:1stmomentT}, Theorem \ref{th:2ndmomentT}, Theorem \ref{th:edgeMarkovT} and Theorem \ref{th:vertexMarkovT}. The proofs are entirely similar to the proof of the corresponding results for invariant processes on $T_d$. We will only sketch the proof and explain the differences.

\subsection{Maximizers of the annealed entropy}

The following lemma is the exact analog of Lemma \ref{th:minSigma} and Lemma \ref{le:staredgetoedgevertex}.
\begin{lemma} \label{th:minSigmaT}
Let $X \in \cI_\pi(M)$ and $r \geq 1$. We have 
$$
\Sigma_{r+1}(X) \leq \Sigma_r (X),
$$
with equality if and only if $X_{S_{r+1}}$ is a $r$-Markov process on $S_{r+1}$. Moreover,
$$
\Sigma_1(X)  \leq \Sigma_e(X), 
$$
with equality if and only if $X_{S_1}$ is a vertex-Markov process on $S_1$. 
\end{lemma}

\begin{proof}
We start by the first statement. Arguing as in the proof of Lemma \ref{th:minSigmaT}, it is enough to check the inequality for $r = 1$. The inequality $\Sigma_2(X) \leq \Sigma_1(X)$ is equivalent to  
\begin{equation}\label{eq:HSHE}
H(X_{S_2}) -  H(X_{S_1}) - \frac d 2 H ( \vec X_{E_2})  + \frac d 2 H ( \vec X_{E_1})\leq 0.
\end{equation}

Let $\deg_T(o)$ and $\deg_{\vec T}(o)$ be the degrees of the root in $T$ and $\vec T$.  For integer $k \geq 1$, 
 from \eqref{eq:vecmu}, we have, for any event $A$, 
$$
\dP ( (\vec X, \vec T) \in A , \deg_{\vec T}(o) = k) = \frac k d \dP ( ( X,  T) \in A, \deg_{T}(o) = k).
$$
It follows that $\dP ( \deg_{\vec T}(o)  = k )  = k \pi(k) /d$ and if $k \geq 1$ is in the support of $\pi$, 
\begin{equation}\label{eq:condlaw}
\dP ( (\vec X, \vec T) \in A | \deg_{\vec T}(o) = k) =  \dP ( ( X,  T) \in A | \deg_{T}(o) = k).
\end{equation}
In other words, $(X, T)$ and  $(\vec X, \vec T)$ have the same law when conditioned on the root degree is $k \geq 1$. 
For a set $S$ of vertices, let us denote by by $H_k ( S)$ the entropy of the variable $(T,X)$ restricted to $S$ and conditioned on the event $\deg_{ T}(o) = k$. Similarly, we set $H_k ( S | S') = H_{k} (S,S') - H_k (S')$ is the associated relative entropy.

From \eqref{eq:propRE}, we may write, for $t = 1,2$, 
$$
H(X_{S_t}) = H ( \deg_T(o) ) + \pi(0) H_0 (X_o) + \sum_{k=1}^\infty \pi(k)  H_k(S_t). 
$$
and 
$$
H(\vec X_{E_t}) = H ( \deg_{\vec T} (o) ) + \sum_{k=1}^\infty \frac{ k \pi(k)}{d} H_k(E_t). 
$$
Hence, \eqref{eq:HSHE} is equivalent to the claim:
\begin{equation}\label{eq:HSHE2}
\sum_{k=1}^\infty \pi(k) \PAR{ H_k( S_2) -  H_k( S_1) - \frac k 2 H_k ( E_2)  + \frac k 2 H_k ( E_1) } \leq 0.
\end{equation}
On the event $\deg_{T}(o) = k$, for $1 \leq i \leq k$, let $L_i = \{(i,1), \ldots, (i,n_i)\}$ be the offspring of vertex $i$. Note that the random variables $X_{L_i}$ conditioned on the event $\deg_{T}(o) = k$ are exchangeable.  
 Then, the computation from \eqref{eq:S2S1b} to \eqref{eq:HSL} gives 
$$
 H_k( S_2) -  H_k( S_1) - \frac k 2 H_k ( E_2)  + \frac k 2 H_k ( E_1) \leq \frac k 2 H_k (L_1|S_1) - \frac k 2 H_k (L |E_1)
$$ 
 The left-hand side of  \eqref{eq:HSHE2} is thus upper bounded by 
$$
\sum_{k=1}^\infty \frac k  2 \pi(k) \PAR{ H_k ( L_1 | S_1)  - H_k ( L | e) } = \frac d 2 \PAR{  H ( \vec X_{L_1} | \vec X_{S_1} ) - H ( \vec X_{L} | \vec X_{E_1} )}. 
$$
We now use the invariance of $\vec \mu$ by switching the two sides of the edge $E_1 = \{ o,1\}$. We get  $H ( \vec X_{L} | \vec X_{E_1} ) =  H ( \vec X_{L_1} | \vec X_{E_1} )$. Finally, since $E_1 \subset S_1$, we deduce that the inequalities \eqref{eq:HSHE}-\eqref{eq:HSHE2} hold.  As in the proof of Lemma \ref{th:minSigma}, the case of equality is a directed consequence of the case of equality in \eqref{eq:propRE}.

We now prove the second statement of Lemma \ref{th:minSigmaT}, $\Sigma_1(X) \leq \Sigma_e(X)$. It is equivalent to prove that 
$$
H(X_{S_1} ) - \frac d 2 H ( \vec X_{E_1}  ) \leq \frac d 2 H(\vec X_{E_1} ) - d H ( \vec X_{o}  )  + H (X_{o}  ). 
$$
Arguing as above, we find that this is equivalent to
$$
\sum_{k =1}^\infty \pi (k) \PAR{ H_k (S_1) - k  H_k (E_1) - (k-1) H_k(o) } \leq 0.
$$
As in the proof of Lemma \ref{le:staredgetoedgevertex}, it remains to use that $H_k (S_1) \leq  H(o) + k H_k (E_1|o)$ and $H_k(E_1|o)  =  H_k(E_1) - H_k(o) $. In the case of equality, this implies the conditional independence of $(X_1,\ldots,X_k)$ given $X_o$ and root degree equal to $k$.
\end{proof}

\subsection{Combinatorial characterization of the annealed entropy}
\label{subsec:SigmaCT}

In our extended setting, the combinatorial interpretation of the annealed entropy $\Sigma_r(\mu)$ explained in Subsection \ref{subsec:SigmaC} continues to hold. The definition of $\Sigma_n(\mu,r,\veps)$ in \eqref{eq:defSigman} remains unchanged. The following theorem is an extension of Theorem \ref{th:combS}.

\begin{theorem}\label{th:combST}
Let $\mu \in \cI_\pi(M)$ and $r \geq 1$ integer. We have
$$
\lim_{\veps \to 0} \liminf_{n \to \infty} \Sigma_n(\mu,r,\veps) = \lim_{\veps \to 0} \limsup_{n \to \infty} \Sigma_n(\mu,r,\veps) = \Sigma_r(\mu).
$$
\end{theorem}

\begin{proof}
Let $s(d) = d/2 - (d/2) \log d$. From \cite[Theorem 2.16]{MR1864966}, we have
$$
 \lim_{n \to \infty} \PAR{ \frac 1 n  \log |\cG_n(d_n) | -  \frac d 2 \log n } = -s(d) + H(\deg(o)) - \dE [ \log (\deg(o) !) ],
$$
where $\deg(o)$ has law $\pi$. Also, from \cite{MR3405616,DA19}, we have 
$$
\lim_{\veps \to 0} \liminf_{n \to \infty} \PAR{ \frac 1 n  \log |\cG_n(\mu,r,\veps) | -  \frac d 2 \log n } = \lim_{\veps \to 0} \limsup_{n \to \infty} \PAR{ \frac 1 n \log |\cG_n(\mu,r,\veps) | -  \frac d 2 \log n } = J_r(\mu),
$$
where $
J_r(\mu) = -s(d) + H(\widetilde X_{S_r} ) - \frac d 2 H (\pi_\mu) - \sum_{g \in \widetilde \cE_r } \dE [\log (N_X (g)!)]
$ is defined exactly as in Theorem \ref{th:combS}, the only difference being that \eqref{eq:sumNX} is replaced by 
$$
\sum_{g \in \widetilde \cE_r } N_X (g) = \deg(o). 
$$
Arguing as in the proof of Theorem \ref{th:combS}, we have $H (\pi_\mu) = H(\widetilde X_{E_r})$ where $\widetilde X_{E_r}$ is the unlabeled coloring associated to $\vec X_{E_r}$. Hence, the theorem follows by checking that
\begin{equation}\label{eq:JrSigma}
H(\widetilde X_{S_r} ) - \frac d 2 H (\widetilde X_{E_r} )  = H(X_{S_r} ) - \frac d 2 H (\vec X_{E_r}) + \sum_{g \in \widetilde \cE_r } \dE [\log (N_X (g)!)]  - \dE [ \log (\deg(o) !) ] .
\end{equation}
In order to prove that \eqref{eq:JrSigma} holds, we decompose the left-hand side of the possible values of the root-degree. We write
$$
H(\widetilde X_{S_r} ) = H ( \pi) + H_0( X_o) + \sum_{k=1}^\infty \pi (k) H_k ( \widetilde X_{S_r} ),
$$
where $H_k$ is the entropy conditioned on $\deg_T(o) = k$. Similarly, 
$$
\frac d 2 H(\widetilde X_{E_r} ) = \frac d 2 H ( \hat \pi) +  \sum_{k=1}^\infty  \frac{k}{2} \pi (k) H_k ( \widetilde X_{E_r} ).
$$ 
Let $\dE_k[\cdot]$ is the expectation conditioned on $\deg_T(o) = k$. From \eqref{eq:condlaw}, it follows that the identity \eqref{eq:JrSigma} is equivalent to 
\begin{equation}\label{eq:JrSigma2}
\sum_{k=1}^\infty \pi (k) \PAR{ H_k ( \widetilde X_{S_r} ) - \frac k 2 H_k ( \widetilde X_{E_r} )} = \sum_{k=1}^\infty \pi (k) \PAR{ H_k (  X_{S_r} ) - \frac k 2 H_k (   X_{E_r} )  +  \sum_{g \in \widetilde \cE_r } \dE_k [\log (N_X (g)!)]  -  \log (k!)  }.
\end{equation}

We denote by $E_r^o$, the tree rooted at $o$  in $E_r \backslash \{o,1\}$ and by $E_r^1$ the tree rooted at $1$. Arguing as in the proof of Theorem \ref{th:combS}, we find 
$$
H_k (   X_{E_r} )  =  H_k ( \widetilde X_{E_r} ) + K_k + K'_k,
$$
where $K_k$ is the relative entropy of a random labeling $\sigma$ of $\widetilde X_{E^o_r}$ conditioned on $\deg(o) = k$ and $K'_k$ is the relative entropy of a random labeling $\sigma'$ of $\widetilde X_{E^1_r}$ conditioned on $\deg(o) = k$.  Similarly, arguing as in the proof of Theorem \ref{th:combS}, we get 
$$
H_k (   X_{S_r} ) = H_k ( \widetilde X_{S_r} )  + k K'_k -  \sum_{g \in \widetilde \cE_r} \dE_k [ \log (N_X(g) !)] +  \log (k!). 
$$
We deduce that the right-hand side of \eqref{eq:JrSigma2} is equal to 
$$
\sum_{k=1}^\infty \pi (k) \PAR{ H_k ( \widetilde X_{S_r} ) - \frac k 2 H_k ( \widetilde X_{E_r} )} + \sum_{k=1}^\infty \pi (k) \PAR{ \frac k 2 K_k' - \frac k 2 K_k}. 
$$
Finally, we observe that 
$$
\sum_{k=1}^\infty \pi (k) \PAR{ \frac k 2 K_k' - \frac k 2 K_k} = \frac d 2 \PAR{ H ( \sigma) - H (\sigma') }.
$$
The above expression is equal to $0$ because the law $\vec \mu$ is invariant by switching the two sides of the edge $\{o,1\}$.  This concludes the proof of Equation \eqref{eq:JrSigma}.
\end{proof}

\begin{remark}\label{rq:suman}
The proof of Theorem \ref{th:combST} gives the simplified expression $J_r(\mu) = -s(d) + \Sigma_r(\mu) - \dE [\log(\deg(o)!)] $. In particular, the proof of Lemma \ref{th:minSigmaT} has an interesting corollary for the maximizers of these functions $J_r(\mu)$. Notably, in \cite[Theorem 1.3]{MR3405616} below (9) we may remove the assumption that $\rho_1$ has finite support. Consequently, from \cite[Corollary 1.4]{MR3405616}, the annealed entropy $\Sigma(\rho)$ defined in \cite{MR3405616} for unimodular random trees with average root degree $d$ is {\em uniquely} maximized by $\UGW(\mathrm{Poi}(d))$. 
\end{remark}

\subsection{Proofs of Theorem \ref{th:1stmomentT}, Theorem \ref{th:2ndmomentT}, Theorem \ref{th:edgeMarkovT} and Theorem \ref{th:vertexMarkovT}}

As already pointed, the conclusion of Theorem \ref{th:concmain} holds in our extended setting. We may thus repeat verbatim the proofs in Section \ref{sec:proofsMain} and invoke Theorem \ref{th:combST} in place of Theorem \ref{th:combS} and Lemma \ref{th:minSigmaT} in place of Lemma \ref{th:minSigma}. \qed

\subsection*{Acknowledgement.} C.B. thanks Suman Chakraborty for his comments and for suggesting Remark \ref{rq:suman}. The research was partially supported by the NKFIH "\'Elvonal" KKP 133921 grant (for \'A. B. and B. Sz.) and by grant ANR-16-CE40-0024 (for C.B.).

\bibliographystyle{abbrv}
\bibliography{bib}

\end{document}